\documentclass[12pt]{article}
\usepackage{amsmath,amssymb,amsthm,mathtools,fullpage,enumerate}
\usepackage[colorlinks=true,linkcolor=red,citecolor=green,urlcolor=magenta]{hyperref}
\usepackage{microtype}
\hypersetup{pdftitle={CM Heights on Unitary Shimura Varieties},pdfauthor={Ziqi Guo}}
\numberwithin{equation}{subsection}
\newtheorem{theorem}{Theorem}[section]
\newtheorem{proposition}[theorem]{Proposition}
\newtheorem{lemma}[theorem]{Lemma}
\theoremstyle{definition}
\newtheorem{remark}[theorem]{Remark}
\theoremstyle{plain}
\newtheorem{corollary}[theorem]{Corollary}
\newcommand{\A}{\mathbb A}
\newcommand{\Af}{\mathbb A_f}
\newcommand{\Q}{\mathbb Q}
\newcommand{\C}{\mathbb C}
\newcommand{\BV}{\mathbb V}
\newcommand{\BW}{\mathbb W}
\newcommand{\LL}{\overline{\mathcal L}}
\newcommand{\UU}{\mathrm U(1,1)}
\newcommand{\U}{\mathrm U}
\newcommand{\Sch}{\mathcal S}
\newcommand{\dd}{\,\mathrm d}
\renewcommand{\Pr}{\mathcal P r}
\DeclareMathOperator{\vol}{vol}
\DeclareMathOperator{\Tr}{Tr}
\DeclareMathOperator{\ord}{ord}
\DeclareMathOperator{\Lie}{Lie}
\DeclareMathOperator{\Fil}{Fil}
\DeclareMathOperator{\CT}{CT}
\DeclareMathOperator{\length}{length}
\DeclareMathOperator{\Int}{Int}
\newcommand{\avint}{\mathop{\rlap{\,--}\!\!\int}\nolimits}
\begin{document}
\title{Heights of small CM cycles on Unitary Shimura Varieties}
\author{Ziqi Guo}
\maketitle

\begin{abstract}
We compute the heights of small CM cycles on unitary Shimura varieties for general choices of the one-dimensional Hermitian subspace $W$ defining the embedding. The formula contains the logarithmic derivative of a quadratic Hecke $L$-function and explicit local terms depending on the choice of $W$. We prove this formula and an arithmetic Siegel--Weil identity by comparing the holomorphic projection of a mixed theta--Eisenstein derivative with arithmetic intersections of special divisors.
\end{abstract}

\begingroup
\small
\tableofcontents
\endgroup
\clearpage

\section{Introduction}

The goal of this paper is to compute the heights of a class of CM points on unitary Shimura varieties. The main idea is to compare the holomorphic projection of a mixed theta--Eisenstein derivative with arithmetic intersections of special divisors. This comparison gives a height formula involving the logarithmic derivative of a quadratic Hecke $L$-function and explicit local terms. We also obtain an arithmetic Siegel--Weil formula.

\subsection{Modular heights of CM points}\label{CM height introduction}

Let $F$ be a totally real field, let $E/F$ be a CM extension, and let $\eta$ be its quadratic Hecke character. We write $L_f(s,\eta)$ for the product of the finite local Hecke $L$-factors. Let $V$ be an $E/F$-Hermitian space of dimension $n+1$, with $n\ge1$, of signature $(n,1)$ at a fixed real place $\iota$ and positive definite at the other real places.

We use the standard lattice $\Lambda$, its stabilizer $U$, and the arithmetic hypotheses of \cite[Assumption~2.4]{Guo2}, recalled in Section~\ref{notation subsection}. Let $\mathcal X_U$ be the normalized integral model of the unitary Shimura variety $X_U$, and let $\LL$ be its metrized Hodge line bundle. These objects are defined in Section~\ref{integral model subsection}.

A one-dimensional subspace $W\subset V$ which is negative at $\iota$ and positive elsewhere defines a normalized CM cycle $P$ of degree one. We denote its normalized arithmetic closure by $\mathcal P$; see \eqref{CM cycle normalization}. For a finite place $v$, let $c_v$ be the nonnegative integer in \eqref{local conductor definition}, determined by the lattice $W_v\cap\Lambda_v$, and let $N_v$ be the cardinality of the residue field of $F_v$. Thus $c_v$ records how $W_v$ meets the fixed lattice at $v$. For $F=\Q$, we assume $S_0(W)=\varnothing$, where $S_0(W)$ is defined in \eqref{S0 definition}.

\begin{theorem}\label{CM height theorem}
For the above data,
\begin{equation*}
    \LL\cdot\mathcal P=-2\frac{L_f'(0,\eta)}{L_f(0,\eta)}
       +\sum_{v\nmid\infty}a_v(c_v)\log N_v,
\end{equation*}
where, for an integer $c\ge0$,
\begin{equation*}
    a_v(c)=\begin{cases}
        c,&v\text{ split in }E/F,\\[1mm]
        \displaystyle c-\frac{2N_v}{N_v^2-1}(1-N_v^{-c}),&v\text{ inert in }E/F,\\[2mm]
        \displaystyle c-\frac{1-N_v^{-c}}{N_v-1},&v\text{ ramified in }E/F.
    \end{cases}
\end{equation*}
\end{theorem}

In fact, for each finite place $v$, define the quadratic splitting symbol
\begin{equation*}
    \chi_v=\begin{cases}
        1,&v\text{ split in }E/F,\\
        -1,&v\text{ inert in }E/F,\\
        0,&v\text{ ramified in }E/F.
    \end{cases}
\end{equation*}
Thus $\chi_v$ agrees with the value of $\eta$ at every unramified place, and the three local terms in Theorem~\ref{CM height theorem} can be written uniformly as
\begin{equation*}
    a_v(c)=c-\frac{1-\chi_v}{(N_v-\chi_v)(1-N_v^{-1})}(1-N_v^{-c}),\qquad c\ge0.
\end{equation*}

The local terms $a_v(c_v)\log N_v$ have a direct relation with the change of Faltings heights under isogenies. For an imaginary quadratic field $K$ and an integer $f\ge1$, Nakkajima--Taguchi \cite{NT} show that the stable Faltings height of an elliptic curve with endomorphism ring $\mathbb Z+f\mathcal O_K$ differs from that for the maximal order $\mathcal O_K$ by $\frac12\sum_{p\mid f}a_p(\ord_p(f))\log p$. Here $a_p$ is given by the three cases above with $N_v=p$ and the splitting of $p$ taken in $K/\Q$. Kudla--Rapoport--Yang \cite[Theorem~10.7]{KRY1} obtain the same local functions, without the factor $1/2$, in their comparison of the Faltings height of a CM abelian surface on a quaternionic Shimura curve over $\Q$ with that of a product of two elliptic curves with CM by $\mathcal O_K$. Their proof uses a calculation of invariant differentials of isogeny kernels \cite[Proposition~10.3]{KRY1}. The canonical and quasi-canonical formal modules in that computation also occur in our local intersection formulas in Section~\ref{local intersection subsection}. All these earlier results concern Faltings heights in the case of imaginary quadratic CM fields, so they do not directly give the degree of the tautological line bundle. Theorem~\ref{CM height theorem} gives the corresponding local corrections for our line bundle over the general CM extensions $E/F$ specified above, for general $W$ in every dimension $n\ge1$. We emphasize that the main theorem also holds when $F=\Q$ under this condition; see Remark~\ref{F=Q remark}.

The CM cycles studied here generalize the \emph{small CM cycles} of \cite[Section~2]{BHKRY} by allowing more general embeddings of the one-dimensional Hermitian subspace $W$ into $V$ relative to the fixed lattice $\Lambda$. Moreover, it is not hard to check from the explicit formula that these points satisfy the following Northcott property.

\begin{corollary}\label{small CM Northcott}
Fix $E/F$, $V$, $\Lambda$ and $U$, and let $\mathcal C_{\mathrm{sm}}$ be the set of small CM cycles arising from the subspaces $W$ in Theorem~\ref{CM height theorem}. For every $B\in\mathbb R$, the set
\begin{equation*}
    \{x\in\mathcal C_{\mathrm{sm}}:h_{\LL}(x)\le B\}
\end{equation*}
is finite.
\end{corollary}

More broadly, the heights of certain special arithmetic 1-cycles on Shimura varieties are often of great importance. For instance, Colmez's conjecture expresses the Faltings height of a CM abelian variety in terms of logarithmic derivatives of Artin $L$-functions \cite{Colmez}. Its averaged form was proved independently by Andreatta--Goren--Howard--Madapusi Pera \cite{AGHMP} and Yuan--Zhang \cite{YZ}. The former use CM cycles on orthogonal Shimura varieties, while the latter relate averages of Faltings heights to heights of CM points on quaternionic Shimura curves. These results show how the arithmetic degree can determine special values of $L$-functions.

More generally, Zhao \cite[Theorem~1.1]{Zhao1} computes heights of special points on quaternionic Shimura varieties in terms of Faltings heights of full CM-types and explicit discriminant terms. His formula allows partial CM-types and extends the quaternionic curve comparison of Yuan--Zhang to higher-dimensional quaternionic Shimura varieties. The line bundle in that formula is the quaternionic Hodge bundle, which agrees with the canonical bundle at sufficiently small level. Our height uses the tautological unitary Hodge line bundle. Its relation with the quaternionic Hodge bundle is given by the integral comparison in \cite[Section~3]{Guo2}. In \cite{Zhao2}, Zhao relates Colmez's period heights to arithmetic degrees of Hermitian line bundles and reduces the full Colmez conjecture to an arithmetic period formula for surfaces. These results concern the relation between period heights and geometric heights; the local terms in Theorem~\ref{CM height theorem} instead measure the integral position of a fixed CM subspace.

For the CM embeddings chosen in \cite{Guo2}, that comparison and \cite[Theorem~1.7]{YZ} give the height formula of \cite[Theorem~1.1]{Guo2}. Theorem~\ref{CM height theorem} allows \textit{every} $W$ with the prescribed signatures on the given integral model. The logarithmic derivative depends only on $E/F$. For the embeddings in \cite[Theorem~1.1]{Guo2} with the present lattice data, every $c_v$ is zero. For a general $W$, the additional terms $a_v(c_v)\log N_v$ measure the change caused by its position relative to $\Lambda$.

\subsection{An arithmetic Siegel--Weil formula}\label{ASW introduction}

Our second result compares two holomorphic cusp forms associated with the CM cycle. The analytic series is the holomorphic projection of a mixed theta--Eisenstein derivative; the arithmetic series records intersections of special divisors with $\mathcal P$.

We write $\A$ and $\A_E$ for the adele rings of $F$ and $E$, respectively, and $\A_f$ for the finite adeles of $F$. In this subsection we further assume $F\ne \mathbb{Q}$. Changing the sign of $W$ at $\iota$ gives a positive incoherent Hermitian space $\BW$ of dimension one. We put $\BV=W^\perp(\A_E)\oplus\BW$ and use the Schwartz function $\Phi$ specified in Section~\ref{split data subsection}. For $\Phi=\phi_1\otimes\phi_2$ with respect to this decomposition, the mixed series is
\begin{equation*}
    I(s,g,\phi_1\otimes\phi_2)=\theta(g,\phi_1)E_{\BW}(s,g,\phi_2),
    \qquad g\in\UU(\A).
\end{equation*}
Here $\theta$ is the theta series of $W^\perp$, and $E_{\BW}$ is the Eisenstein series of $\BW$. We extend the definition by linearity and denote the cuspidal holomorphic projection of $I'(0,g,\Phi)$ by $\Pr I'(0,g,\Phi)$. We call this the \textit{derivative series}. Section~\ref{mixed subsection} gives the definitions and normalizations.

Let $\widehat{\mathcal Z}_*(g,\Phi)$ be the positive Fourier part of the normalized admissible arithmetic generating series of special divisors, with generic fiber $Z_*(g,\Phi)$. The \textit{height series} is the pairing
\begin{equation*}
    \widehat\xi=\frac{\widehat c_1(\LL)^n}{\deg_L X_U},\qquad
    \mathcal H(g,\Phi)=\left(\widehat{\mathcal Z}_*(g,\Phi)
      -\frac{\deg_L Z_*(g,\Phi)}{\deg_L X_U}\LL\right)
      \cdot(\mathcal P-\widehat\xi).
\end{equation*}
The arithmetic extensions and pairings are defined in Section~\ref{integral model subsection}.

For a finite nonsplit place $v$, let $W(v)$ be obtained from $\BW$ by changing its invariant at $v$, and put $V(v)=W^\perp\perp W(v)$. We use the local difference
\begin{equation*}
    \bar k_{\Phi_v}(1,y)=2k_{\Phi_v}(1,y)-m_{\Phi_v}(y)\log N_v,
    \qquad y\notin W_v^\perp,
\end{equation*}
where $k_{\Phi_v}$ is the Whittaker derivative in \eqref{local k definition} and $m_{\Phi_v}\log N_v$ is the proper local intersection defined in Section~\ref{standard finite subsection}. Section~\ref{local intersection subsection} proves that this difference extends to a Schwartz function on $V(v)_v$.

The theta series below also uses the action of the norm-one torus $E^1$. For $\tau\in E^1(\A)$, let $\tau$ act on $W(v)$ and trivially on $W^\perp$, and write $r_{V(v)}(g,\tau)$ for this action together with the Weil representation. For $\varphi\in\Sch(V(v)(\A_E))$, put
\begin{equation*}
    \theta(g,\tau,\varphi)=\sum_{y\in V(v)(E)}r_{V(v)}(g,\tau)\varphi(y).
\end{equation*}
Here $\Sch$ denotes the Schwartz space. We write $\avint_{[E^1]}$ for the average over $[E^1]=E^1(F)\backslash E^1(\A)$, and a superscript $v$ omits the component at $v$. Let $K_f\subset\UU(\A_f)$ be the sufficiently small compact open subgroup chosen at the end of Section~\ref{split data subsection}.

\begin{theorem}\label{ASW main theorem}
With the notation above, for $g\in\UU(\A)$ with $g_f\in K_f$,
\begin{equation}\label{ASW main identity}
    \Pr I'(0,g,\Phi)+\mathcal H(g,\Phi)
    =-\sum_{\substack{v\nmid\infty\\v\ \mathrm{nonsplit}}}
      \avint_{[E^1]}\theta\bigl(g,\tau,\bar k_{\Phi_v}\otimes\Phi^v\bigr)\,\dd\tau.
\end{equation}
Moreover, only finitely many places $v$ contribute to the sum.
\end{theorem}

The arithmetic Siegel--Weil formula relates arithmetic intersections of special cycles to derivatives of Eisenstein series. It is a central part of Kudla's program \cite{Kudla}, which extends the relation between geometric degrees and values of Eisenstein series. On Shimura curves, the work of Kudla--Rapoport--Yang \cite{KRY} establishes such identities for arithmetic generating series. The Gross--Zagier formula of Yuan--Zhang--Zhang \cite{YZZ} and the CM height formula of Yuan--Zhang \cite{YZ} also arise from comparisons of analytic and arithmetic generating series. Yuan \cite{Yuan} uses the comparison of these series to compute modular heights of quaternionic Shimura curves.

On orthogonal Shimura varieties, Bruinier--Yang \cite{BY} compute the archimedean contribution to pairings of arithmetic Heegner divisors with CM cycles and prove the predicted relation with $L$-derivatives in several low-dimensional cases. For unitary Shimura varieties, Liu \cite{Liu1,Liu2} relates arithmetic theta lifts to central derivatives of $L$-functions. Bruinier--Howard--Yang \cite{BHY} study height pairings of special divisors, and Qiu \cite{Qiu} proves modularity for arithmetic special divisors with admissible extensions. The comparison in \cite[Theorem~1.3]{Guo2} treats a particular choice of CM embedding. Formula~\eqref{ASW main identity} gives the comparison for general $W$, including the theta series arising from the finite local differences.

It is worth noting that the reason we assume $F\neq\mathbb{Q}$ here is that Proposition \ref{cuspidal height pairing} does not directly hold when $F=\mathbb{Q}$, since we require the integral model to be proper. All the other local computations in this paper, however, remain unchanged in the case $F=\mathbb{Q}$.

\subsection{Idea of proof and organization}

For $F\ne\Q$, we prove both theorems by computing the derivative series and the height series separately. Their sum is a holomorphic cusp form. Combining the Whittaker derivatives with the arithmetic intersection numbers for $y\notin W^\perp$, we obtain theta series on the spaces $V(v)$. The remaining contribution is a scalar multiple of the sum over nonzero vectors in $W^\perp(E)$. The Key lemma of \cite[Lemma~6.1]{YZ} and \cite[Lemma~2.3]{Yuan}, in the unitary form of \cite[Lemma~2.4]{Guo1}, determines this scalar and gives Theorem~\ref{CM height theorem}. Substitution in the same identity then proves Theorem~\ref{ASW main theorem}.

In Section~\ref{analytic section}, we define the derivative series and state its decomposition in Theorem~\ref{derivative decomposition theorem}. We choose functions at two split places whose Siegel sections for $\BV$ vanish and whose restriction to $W^\perp$ has a nonzero theta series. Then we decompose the local lattices relative to $W$ and compute the zero and nonzero Whittaker coefficients. These compute explicit expressions of the functions $c_{\Phi_v}$ and $k_{\Phi_v}$ in the derivative series.

In Section~\ref{height series section}, we define the height series and decompose it in Theorem~\ref{height decomposition theorem}. The proper terms are computed from the lifting lengths of homomorphisms between canonical and quasi-canonical formal modules. We then compute the degrees of the divisor line bundles restricted to the CM cycle. At a ramified place where $W\cap\Lambda$ is self-dual, we use the dual lattice to identify the relevant lifting divisor and compare its normal bundle with the given Hodge bundle.

In Section~\ref{comparison section}, we add the formulas from the two preceding sections. Proposition~\ref{direct comparison proposition} records the cancellations, and the Key lemma gives the two main theorems. Finally, we explain the case when $F=\Q$ in Remark \ref{F=Q remark}.

\subsection*{AI Declaration}
AI tools (GPT-6) were only used to assist with writing and verification of this paper. All the ideas, structure, and proof strategies were developed by the author. The author takes full responsibility for every statement in this paper.

\subsection*{Acknowledgement}
The author is deeply grateful for the valuable assistance and meticulous guidance provided by Professor Xinyi Yuan. He thanks Weixiao Lu, Jiacheng Xia, Yu Luo and Roy Zhao for helpful communication. He also thanks Chao Li, Yifeng Liu, Tonghai Yang and Wei Zhang for some useful suggestions. Finally, the author is grateful to the anonymous referee for many valuable comments and suggestions on this paper.

\section{Derivative series}\label{analytic section}

In this section, we define the derivative series $\Pr I'(0,g,\Phi)$ and express it as a sum of local terms in Theorem~\ref{derivative decomposition theorem}. Next, we choose $\Phi$ and compute those terms. The zero Whittaker coefficients give $c_{\Phi_v}$, while the nonzero coefficients give $k_{\Phi_v}$, which will be compared with the proper intersections in Section~\ref{local intersection subsection}.

\subsection{Notation and normalizations}\label{notation subsection}

This subsection fixes the lattice and local notation used throughout the paper, together with the measures and Weil representations of \cite[Section~2.1]{Guo1}.

The fields $E/F$ and the spaces $V,W$ are those of Section~\ref{CM height introduction}, with distinguished real place $\iota$. We write $d=[F:\Q]$, $F_{>0}$ for the totally positive elements of $F$, and $\Af$ for the finite adeles of $F$. The absolute discriminant of $F$ is $d_F$, and $d_{E/F}$ is the norm of the relative discriminant of $E/F$. We write $q(x)=\langle x,x\rangle$, with the Hermitian pairing linear in its first variable. The positive incoherent space in Section~\ref{ASW introduction} is
\begin{equation*}
    \BV=W^\perp(\A_E)\oplus\BW.
\end{equation*}
At finite places this agrees with $V_v=W_v^\perp\perp W_v$.

At a finite place $v$, let $\varpi_v$ be a uniformizer and recall that $N_v$ is the residue cardinality. We use $v(a)=\ord_v(a)$ and $|a|_v=N_v^{-v(a)}$, with $v(\varpi_v)=1$, and put $|a|_f=\prod_{v\nmid\infty}|a|_v$. At a ramified place of $E/F$, we choose $\varpi_{E_v}$ with $\overline{\varpi}_{E_v}=-\varpi_{E_v}$. At a split place, we identify $E_v=F_v\times F_v$; conjugation exchanges the two coordinates, and $\mathcal O_{E_v}=\mathcal O_{F_v}\times\mathcal O_{F_v}$. All split coordinates below use this identification.

We impose \cite[Assumption~2.4]{Guo2}. Thus $E$ is Galois over $\Q$ or is the compositum of $F$ with an imaginary quadratic field; every place above $2$ splits in $E/F$; and every place ramified over $\Q$ is unramified in $E/F$. The standard lattice $\Lambda$ and its stabilizer $U$ are those of \cite[Definition~2.2]{Guo1}. In particular, $\Lambda_v$ is self-dual at split and inert places. At a ramified place it is $\varpi_{E_v}$-modular for odd $n$ and almost $\varpi_{E_v}$-modular for even $n$.

We take lattice duals for the Hermitian pairing:
\begin{equation*}
    \Lambda_v^\vee=\{x\in V_v:\langle x,\Lambda_v\rangle\subset\mathcal O_{E_v}\}.
\end{equation*}
The $\varpi_{E_v}$-modular condition is $\Lambda_v^\vee=\varpi_{E_v}^{-1}\Lambda_v$. In the almost $\varpi_{E_v}$-modular case, $\Lambda_v^\vee\subset\varpi_{E_v}^{-1}\Lambda_v$ has colength one; see \cite[Section~4.4]{Qiu}.

For the embedding of $W$, we put
\begin{equation}\label{CM local lattices}
    M_v=W_v\cap\Lambda_v=\mathcal O_{E_v}w_v,\qquad
    \Lambda_v^\perp=W_v^\perp\cap\Lambda_v,\qquad r_v=v(q(w_v)).
\end{equation}
Here $M_v$ is a full $\mathcal O_{E_v}$-lattice in $W_v$, and $w_v$ is a generator. Since $M_v=W_v\cap\Lambda_v$, this generator is primitive in $\Lambda_v$; at a split place, the same holds in both coordinates. The integer in Theorem~\ref{CM height theorem} is
\begin{equation}\label{local conductor definition}
    c_v=\begin{cases}
        r_v,&v\text{ split or inert},\\
        \max\{r_v-1,0\},&v\text{ ramified}.
    \end{cases}
\end{equation}
The integers $c_v$ depend only on $W$ and $\Lambda$, and vanish at all but finitely many places. We also put
\begin{equation}\label{S0 definition}
    S_0(W)=\{v\nmid\infty:v\text{ ramified in }E/F,\ r_v=0\}.
\end{equation}
This set is empty for odd $n$. At its places, $M_v$ is self-dual and $\Lambda_v=M_v\perp\Lambda_v^\perp$, with $\Lambda_v^\perp$ $\varpi_{E_v}$-modular; see Lemma~\ref{all place line types}.

For the analytic normalizations, we follow \cite[Section~2.1]{Guo1}. The additive character at $v$ is the standard character of $\Q_p$ composed with $\Tr_{F_v/\Q_p}$, and $d_v$ is the different of $F_v/\Q_p$. We denote the relative discriminant of $E_v/F_v$ by $D_v$. With the self-dual measures of \cite[Section~1.6.2]{YZZ} and \cite[Section~2.1]{Guo1}, the volume of $\mathcal O_{F_v}$ is $|d_v|_v^{1/2}$. On a Hermitian space, the measure is self-dual for $q$ and $\psi_v$. At a real place, $\psi_v(b)=e^{2\pi i b}$ and the standard Gaussian is $e^{-2\pi q}$. We write $\psi=\prod_v\psi_v$.

We write $\Sch(T)$ for the Schwartz space of a Hermitian space $T$ and $r_T$ for its Weil representation, with the splitting characters of \cite[Section~2.1]{Guo1}. The upper triangular parabolic of $\UU$ is $\mathsf P$. We use the matrices
\begin{equation*}
    n(b)=\begin{pmatrix}1&b\\0&1\end{pmatrix},\qquad
    m(a)=\begin{pmatrix}a&0\\0&\bar a^{-1}\end{pmatrix},\qquad
    w=\begin{pmatrix}0&1\\-1&0\end{pmatrix},
    \qquad b\in F_v,\ a\in E_v^\times,
\end{equation*}
and the Iwasawa function
\begin{equation*}
    \delta_v(m(a))=|a|_{E_v}^{1/2},\qquad
    \delta(g)=\prod_v\delta_v(g_v),\qquad
    \delta_f(g_f)=\prod_{v\nmid\infty}\delta_v(g_v).
\end{equation*}
We omit the subscript of $r_T$ when the space is determined by its argument. A superscript $v$ omits the component at $v$.

The measures on the norm-one torus are those of \cite[Section~2.4]{Qiu}. For $[E^1]=E^1(F)\backslash E^1(\A)$, they give
\begin{equation*}
    \vol([E^1])=2L(1,\eta),\qquad
    \avint_{[E^1]}f(\tau)\,\dd\tau
       =\frac{1}{2L(1,\eta)}\int_{[E^1]}f(\tau)\,\dd\tau,
\end{equation*}
where $L(s,\eta)$ has the completed normalization of \cite[Section~3.1]{Guo1}. At a finite nonsplit place,
\begin{equation*}
    \vol(E_v^1)=\begin{cases}
        N_v^{-v(d_v)/2},&v\text{ inert},\\
        2N_v^{-1/2},&v\text{ ramified}.
    \end{cases}
\end{equation*}
The hypotheses on $E/F$ give $v(d_v)=0$ in the ramified case. We write $r(g,\tau)$ for the Weil action together with the action of $\tau$ on the one-dimensional factor. For a factorizable Schwartz function, we also write $\Phi_v^\perp=\Phi_v|_{W_v^\perp}$ and $\Phi^\perp=\otimes_v\Phi_v^\perp$.

\subsection{Mixed theta--Eisenstein series}\label{mixed subsection}

We use the construction of \cite[Section~7.1]{YZ} and \cite[Section~3.2]{Qiu}, with the normalization of \cite[Section~3.1]{Guo1}. Let $\Phi=\otimes_v\Phi_v$ be standard at infinity and equal to $1_{\Lambda_v}$ at almost every finite place. A partition into product cosets expresses each finite component as a finite sum of tensors relative to $W_v^\perp\perp W_v$. Note that the standard Gaussian at infinity already factors with respect to $W^\perp\oplus\BW$.

Define the mixed theta--Eisenstein series by
\begin{equation*}
    I(s,g,\Phi)=\sum_{\gamma\in\mathsf P(F)\backslash\UU(F)}\delta(\gamma g)^s\sum_{x\in W^\perp(E)}r_{\BV}(\gamma g)\Phi(x,0).
\end{equation*}
For $\Phi=\phi_1\otimes\phi_2$, this is $\theta(g,\phi_1)E_{\BW}(s,g,\phi_2)$, where
\begin{equation*}
    \theta(g,\phi_1)=\sum_{x\in W^\perp(E)}r_{W^\perp}(g)\phi_1(x),\qquad E_{\BW}(s,g,\phi_2)=\sum_{\gamma\in\mathsf P(F)\backslash\UU(F)}\delta(\gamma g)^sr_{\BW}(\gamma g)\phi_2(0).
\end{equation*}
Finite tensor decompositions give the meromorphic continuation of $I(s,g,\Phi)$ and its derivative $I'(0,g,\Phi)$. We also use
\begin{equation*}
    E_{\BV}(s,g,\Phi)=\sum_{\gamma\in\mathsf P(F)\backslash\UU(F)}\delta(\gamma g)^sr_{\BV}(\gamma g)\Phi(0)
\end{equation*}
and its positive Fourier part $E_{\BV,*}(0,g,\Phi)$, with the normalization of \cite{Guo1}. We also denote the sum over nonzero vectors in the definition of $I(s,g,\Phi)$ by
\begin{equation}\label{Omega definition}
    \Omega_\Phi(g)=\sum_{0\ne x\in W^\perp(E)}r_{\BV}(g)\Phi(x,0).
\end{equation}

For $a\in F_v$ and $\phi_{2,v}\in\Sch(\BW_v)$, let
\begin{equation*}
    W_{a,v}(s,g_v,\phi_{2,v})=\int_{F_v}\delta_v(wn(b)g_v)^s r_{\BW_v}(wn(b)g_v)\phi_{2,v}(0)\psi_v(-ab)\,\dd b.
\end{equation*}
For $a\ne0$, we put $W_{a,v}^\circ=\gamma_v^{-1}W_{a,v}$, where $\gamma_v$ is the Weil index of $\BW_v$. At a finite place, we normalize the coefficient at zero by
\begin{equation*}
    W_{0,v}^\circ(s,g_v,\phi_{2,v})=\gamma_v^{-1}\frac{L_v(s+1,\eta_v)}{L_v(s,\eta_v)}|D_vd_v|_v^{-1/2}W_{0,v}(s,g_v,\phi_{2,v}).
\end{equation*}
The ideals $D_v,d_v$ and their absolute values have the normalization of Section~\ref{notation subsection}. At infinity, we use the completed normalization of \cite[Section~3.1]{Guo1}.

Next, we define the local functions that occur in the derivative series. For a finite nonsplit place $v$ or a real place $v$, write $W(v)$ for the coherent Hermitian space of dimension one obtained by changing the invariant of $\BW$ at $v$, and put $V(v)=W^\perp\perp W(v)$. At every finite place, for a tensor $\Phi_v=\phi_{1,v}\otimes\phi_{2,v}$, we define
\begin{equation}\label{c definition}
    c_{\Phi_v}(g_v,y_1)=r_{W_v^\perp}(g_v)\phi_{1,v}(y_1)(W_{0,v}^\circ)'(0,g_v,\phi_{2,v})+\log\delta_v(g_v)r_{\BV_v}(g_v)\Phi_v(y_1,0).
\end{equation}
If $v$ is finite and nonsplit, and $y=y_1+y_2\in V(v)_v$ with $y_2\ne0$, we put
\begin{equation}\label{local k definition}
    k_{\Phi_v}(g_v,y)=\frac{L_v(1,\eta_v)}{\vol(E_v^1)}r_{W_v^\perp}(g_v)\phi_{1,v}(y_1)(W_{q(y_2),v}^\circ)'(0,g_v,\phi_{2,v}).
\end{equation}
These definitions extend by linearity to every $\Phi_v$. For $a=q(y_2)$ in the definition of $k_{\Phi_v}$, the local invariant gives
\begin{equation*}
    \eta_v\bigl(a/q(w_v)\bigr)=-1.
\end{equation*}

The finite local contribution is obtained from the sums
\begin{equation*}
    \begin{aligned}
    \mathcal K_\Phi^{(v)}(g,\tau)&=\sum_{y\in V(v)(E)\setminus W^\perp(E)}k_{r(\tau_v)\Phi_v}(g_v,y)r(g^v,\tau^v)\Phi^v(y),\\
    I'(0,g,\Phi)(v)&=2\avint_{[E^1]}\mathcal K_\Phi^{(v)}(g,\tau)\,\dd\tau.
    \end{aligned}
\end{equation*}

At an archimedean place $v$, the space $V(v)$ has signature $(n,1)$ at $v$. For $y=y_1+y_2$ with $q(y)$ totally positive and $y_2\ne0$, put
\begin{equation}\label{archimedean k}
    k_{v,u}(y)=\frac{\Gamma(n+u)}{2(4\pi)^u\Gamma(n)}\int_1^\infty\frac{\dd b}{b\left(1-\frac{q_v(y_2)}{q_v(y)}b\right)^{n+u}}.
\end{equation}
We now form the archimedean contribution to the derivative series. For a meromorphic function of $u$, we write $\CT_{u=0}$ for the coefficient of $u^0$ in its Laurent expansion at zero. For each $t\in F_{>0}$, the sum below has the meromorphic continuation of \cite[Section~3.3.7]{Qiu}, with the kernel normalization of \cite[Section~3.2]{Guo1}. Define
\begin{equation*}
    \begin{aligned}
    \overline{\mathcal K}_\Phi^{(v)}(g,\tau)&=\sum_{t\in F_{>0}}\CT_{u=0}\sum_{\substack{y\in V(v)(E)\setminus W^\perp(E)\\q(y)=t}}k_{v,u}(y)r(g,\tau)\Phi(y),\\
    \overline I'(0,g,\Phi)(v)&=2\avint_{[E^1]}\overline{\mathcal K}_\Phi^{(v)}(g,\tau)\,\dd\tau.
    \end{aligned}
\end{equation*}
At the real place $v$, $r(g,\tau)\Phi(y)$ denotes the holomorphic Whittaker factor determined by the standard Gaussian and $t$; at the other places it denotes the usual Weil action.

\subsubsection*{Holomorphic projection}

We work separately on the finitely many components for the central character at the fixed level, as in \cite[Section~7.1]{YZ}. The standard Gaussian has scalar weight $n+1$ on $\mathrm{SU}(1,1)$. The cuspidal holomorphic projection $\Pr I'$ is characterized by
\begin{equation*}
    \langle\Pr I',h\rangle_{\mathrm{Pet}}=\langle I',h\rangle_{\mathrm{Pet}}
\end{equation*}
for every holomorphic cusp form $h$ of the same weight and central character. The pairing converges by rapid decay of $h$.

To compute this projection, we recall the definition of the normalized quasi-holomorphic projection $\Pr'$ from \cite[Section~3.2]{Guo1}. Let $f$ be a smooth automorphic function of moderate growth with the fixed archimedean weight $\mathfrak m$ and central character. Its Whittaker coefficient at $t\in F_{>0}$ is
\begin{equation*}
    f_t(g)=\int_{F\backslash\A}f(n(b)g)\psi(-tb)\,\dd b.
\end{equation*}
Write $F_\infty=\prod_{v\mid\infty}F_v$, and let $\mathsf Z$ and $\mathsf N$ be the center and upper triangular unipotent subgroup of $\UU$. The holomorphic Whittaker function $W_t^{(\mathfrak m)}$ has this weight and central character, with normalization
\begin{equation*}
    \int_{\mathsf Z(F_v)\mathsf N(F_v)\backslash\UU(F_v)}
       |W_{t,v}^{(\mathfrak m)}(h)|^2\,\dd h
       =(4\pi)^{-n}\Gamma(n),\qquad v\mid\infty.
\end{equation*}
For $|t|_\infty=\prod_{v\mid\infty}t_v$, put
\begin{equation*}
    \begin{aligned}
    f_{t,s}(g)={}&\frac{(4\pi)^{dn}}{\Gamma(n)^d}|t|_\infty^s
      \left(\frac{\Gamma(n)(4\pi)^s}{\Gamma(n+s)}\right)^{d-1}
      W_t^{(\mathfrak m)}(g_\infty)\\
    &\quad\cdot\int_{\mathsf Z(F_\infty)\mathsf N(F_\infty)\backslash\UU(F_\infty)}
      \delta(h)^{2s}f_t(g_fh)\overline{W_t^{(\mathfrak m)}(h)}\,\dd h.
    \end{aligned}
\end{equation*}
When this family has a meromorphic continuation to $s=0$, define
\begin{equation}\label{quasi holomorphic projection definition}
    \Pr'(f)_t(g)=\CT_{s=0}f_{t,s}(g),\qquad
    \Pr'(f)(g)=\sum_{t\in F_{>0}}\Pr'(f)_t(g).
\end{equation}
This is the construction of \cite[Definition~6.3]{Liu2} with the normalization in \cite[Section~3.2]{Guo1}. The equality with $\Pr(f)$ under the required growth condition follows from \cite[Proposition~6.2]{Liu2}, in the unitary form of \cite[Lemma~3.3.1]{Qiu}.

In addition, we impose the following condition at two distinct finite split places $S_2$:
\begin{equation}\label{zero local section}
    r_{\BV_v}(g_v)\Phi_v(0)=0\qquad(v\in S_2,\ g_v\in\UU(F_v)).
\end{equation}
Section~\ref{split data subsection} constructs these functions explicitly. Every section of $E_{\BV}(s,g,\Phi)$ contains two zero factors, so the series, its derivative, and every series obtained by changing the datum at one place vanish identically.

\begin{proposition}\label{projection equality}
For every $n\ge1$ and a datum satisfying \eqref{zero local section},
\begin{equation*}
    \Pr I'(0,g,\Phi)=\Pr'I'(0,g,\Phi).
\end{equation*}
\end{proposition}
\begin{proof}
We first remove the Eisenstein term $\mathcal J$ formed from the absolute constant term in \cite[Section~3.2]{Guo1}. Its derivative is a linear combination of the Eisenstein series for $\BV$, its derivative, and series with a section changed at one place. Equation~\eqref{zero local section} therefore gives $\mathcal J'=0$.

The growth calculation of \cite[Lemma~6.13]{YZZ} and \cite[Lemma~3.3.3]{Qiu}, used in \cite[Section~3.2]{Guo1}, gives
\begin{equation*}
    (I'-\mathcal J')(m(a_v)g)
    =O_g\bigl(|a_v|_{E_v}^{(1-n)/2+\epsilon}\bigr).
\end{equation*}
For $0<\epsilon<1/2$, this satisfies the projection criterion for every $n\ge1$. The positive coefficient families and the archimedean kernel sums have the continuation specified in \cite[Section~3.2]{Guo1}; here $[F:\Q]>1$. At $n=1$, the coefficient integrals for $I'-\mathcal J'$ have a finite limit at zero by the same argument, so their singularity is removable. The definition \eqref{quasi holomorphic projection definition} therefore gives
\begin{equation*}
    \Pr I'=\Pr'(I'-\mathcal J')=\Pr'I'.
\end{equation*}
\end{proof}

To state the derivative formula, we write $H_j=\sum_{a=1}^j a^{-1}$, with $H_0=0$, and put
\begin{equation*}
    c_1=2\frac{L_f'(0,\eta)}{L_f(0,\eta)}+\log d_F+\log d_{E/F}.
\end{equation*}
The following formula uses the local terms defined above. Section~\ref{local Whittaker subsection} computes their values.

\begin{theorem}\label{derivative decomposition theorem}
For every $n\ge1$ and every Schwartz function satisfying the conditions of this subsection,
\begin{equation}\label{analytic formula}
    \begin{aligned}
    \Pr I'(0,g,\Phi)={}&-\sum_{v\mid\infty}\overline I'(0,g,\Phi)(v) -\sum_{\substack{v\nmid\infty\\v\ \mathrm{nonsplit}}}I'(0,g,\Phi)(v)\\
    &+(-c_1+dH_{n-1})\Omega_\Phi(g) -\sum_{v\nmid\infty}\sum_{0\ne x\in W^\perp(E)} c_{\Phi_v}(g_v,x)r(g^v)\Phi^v(x,0)\\
    &+\sum_{0\ne x\in W^\perp(E)} \bigl(2\log\delta_f(g_f)+\log|q(x)|_f\bigr)r(g)\Phi(x,0).
    \end{aligned}
\end{equation}
Only finitely many places contribute to the sum of the $c_{\Phi_v}$ terms.
\end{theorem}
\begin{proof}
We apply the Fourier decomposition of \cite[Section~7.1]{YZ} and \cite[Section~3.2]{Qiu} to each tensor, and then use the holomorphic projection formula in \cite[Theorem~3.3]{Guo1}. The same argument applies here, since $W^\perp$ is coherent and totally positive definite. Proposition~\ref{projection equality} then identifies the projection with $\Pr I'$.

More precisely, the archimedean computation in \cite[Theorem~3.3]{Guo1} gives the function $k_{v,u}$ in \eqref{archimedean k} and the harmonic term $dH_{n-1}$. Its discriminant term is $\log d_F+\log d_{E/F}$, which is included in $c_1$, and the product formula gives $\log|q(x)|_f$. These formulas hold for all $n\ge1$, with the convention that $H_0=0$.
\end{proof}

\subsection{Choice of the Schwartz function}\label{split data subsection}

We now choose $\Phi$ for the decomposition in Theorem~\ref{derivative decomposition theorem}. Following \cite[Section~7.2]{YZ}, we require \eqref{zero local section} and a nonzero theta series on $W^\perp$. Recall that $M_v=W_v\cap\Lambda_v$ and $\Lambda_v^\perp=W_v^\perp\cap\Lambda_v$, as in \eqref{CM local lattices}.

Throughout this paper, we choose two split places $S_2=\{v_1,v_2\}$ above distinct odd rational primes, unramified over $\Q$, such that $\psi_v$ has conductor $\mathcal O_{F_v}$ and $\Lambda_v=\Lambda_v^\perp\perp M_v$. Chebotarev gives infinitely many split places, and orthogonal projection onto $W$ is integral outside a finite set, so these choices are possible. For $v\in S_2$, put
\begin{equation}\label{degenerate phi}
    \Phi_v(y)=\left(1_{\varpi_v\Lambda_v}(y)-\frac{1_{\Lambda_v}(y)}{1+N_v^n+N_v^{2n}}\right)1_{\varpi_v^2\mathcal O_{F_v}^\times}(q(y)).
\end{equation}
This is a higher-dimensional generalization of the choice in \cite[Section~7.2]{YZ}. At $v\in S_0(W)$, take
\begin{equation}\label{unit test function}
    \Phi_v=1_{\Lambda_v^\vee}
    =1_{(\Lambda_v^\perp)^\vee}\otimes1_{M_v}.
\end{equation}
Here $M_v$ is self-dual and $(\Lambda_v^\perp)^\vee=\varpi_{E_v}^{-1}\Lambda_v^\perp$.
This choice allows the divisors in the height series to be described by lifting homomorphisms; the reason is given in Lemma~\ref{unit horizontal divisor}. At the remaining finite places take $\Phi_v=1_{\Lambda_v}$, and at infinity take the standard Gaussian.

We put $\Phi=\otimes_v\Phi_v$ and $\Phi_f^\perp=\otimes_{v\nmid\infty}\Phi_v^\perp$. Each finite component is $U_v$-invariant. Section~\ref{local line conductor subsection} gives its decomposition relative to $W_v^\perp\perp W_v$.

The measures fixed in Section~\ref{notation subsection} give the standard self-dual lattices volume one at $v\in S_2$, because $F_v/\Q_p$ is unramified. With these measures, we verify the two required properties of \eqref{degenerate phi}.

\begin{proposition}\label{two place proposition}
For $v\in S_2$ and $g_v\in\UU(F_v)$,
\begin{equation*}
    r_{\BV_v}(g_v)\Phi_v(0)=0.
\end{equation*}
Moreover,
\begin{equation*}
    \int_{W_v^\perp}\Phi_v^\perp(y_1)\,\dd y_1=N_v^{-2n}(1-N_v^{-1})(1-N_v^{-n})\left(1-\frac{1+N_v^{n-1}+N_v^{2n-2}}{1+N_v^n+N_v^{2n}}\right)>0.
\end{equation*}
Consequently the Eisenstein series for $\BV$ attached to $\Phi$ vanishes identically, and the same is true after changing the datum at any one place.
\end{proposition}
\begin{proof}
The two assertions follow by counting the same norm congruences in $V_v$ and $W_v^\perp$. First consider $V_v$, with split coordinates $z,z'\in\mathcal O_{F_v}^{n+1}$ and $q(z,z')=z^tz'$.

For $a\ne0$ with $v(a)\ge0$ and an integer $\ell>v(a)$, the vectors $z$ whose coordinates have minimum valuation $j$ have volume $(1-N_v^{-n-1})N_v^{-j(n+1)}$. For such $z$, the condition $z^tz'\in a+\varpi_v^\ell\mathcal O_{F_v}$ has volume $N_v^{j-\ell}$ in $z'$ when $j\le v(a)$, and has no solution otherwise. Hence
\begin{equation*}
    \begin{aligned}
    &\vol\{(z,z')\in\mathcal O_{F_v}^{n+1}\oplus\mathcal O_{F_v}^{n+1}:
             z^tz'\in a+\varpi_v^\ell\mathcal O_{F_v}\}\\
    &\qquad=N_v^{-\ell}(1-N_v^{-n-1})
          \sum_{j=0}^{v(a)}N_v^{-jn}.
    \end{aligned}
\end{equation*}
When $v(a)=2$, division by $\varpi_v$ in both coordinates shows that the ratio of this volume on $\Lambda_v$ to its volume on $\varpi_v\Lambda_v$ is $1+N_v^n+N_v^{2n}$. Thus the two terms in \eqref{degenerate phi} have equal integrals over every sufficiently small norm congruence class. In particular,
\begin{equation*}
    r_{\BV_v}(wn(b))\Phi_v(0)
       =\int_{\BV_v}\Phi_v(y)\psi_v(bq(y))\,\dd y=0.
\end{equation*}
Since $\Phi_v(0)=0$, the Bruhat decomposition proves \eqref{zero local section}.

For the restriction to $W_v^\perp$, each split coordinate has $n$ entries. The same count gives volume $(1-N_v^{-1})(1-N_v^{-n})$ for vectors of unit norm, and the ratio at norm valuation two is $1+N_v^{n-1}+N_v^{2n-2}$. After scaling by $\varpi_v$, these values give the integral in the statement. Finally, changing the datum at one place leaves a zero Siegel section at the other place of $S_2$, so the asserted Eisenstein series vanish.
\end{proof}

We will also use the following non-vanishing result of the restricted theta series.

\begin{lemma}\label{nonzero rational theta coefficient}
For every sufficiently small real number $a>0$,
\begin{equation*}
    \sum_{x\in W^\perp(E)}\Phi_f(x,0)\exp\bigl(-2\pi a^2\Tr_{F/\Q}q(x)\bigr)>0.
\end{equation*}
\end{lemma}
\begin{proof}
The idea is to use Poisson summation to express the sum in terms of its Fourier transform. Note that the finite restrictions at the standard places and the places of \eqref{unit test function} have positive integral, so Proposition~\ref{two place proposition} gives $\int\Phi_f^\perp>0$. By Poisson summation, the sum multiplied by $a^{2nd}$ tends to a positive constant times this integral as $a\to0^+$. Since the integral is positive, the sum is also positive for sufficiently small $a$.
\end{proof}

From now on, $\Phi$ denotes this choice. We will compute at $g_f=1$ and choose a compact open subgroup $K_f$ of the standard maximal compact subgroup fixing $\Phi_f$, $\Phi_f^\perp$, and the finite Schwartz functions in the comparison. The same formulas then hold for $g_f\in K_f$.

\subsection{Local lattices and conductors}\label{local line conductor subsection}

To compute the functions $c_{\Phi_v}$ and $k_{\Phi_v}$ in Theorem~\ref{derivative decomposition theorem}, we express the lattice function as a sum of tensors on $W_v^\perp\perp W_v$. The same description identifies the quadratic order needed for the lifting lengths in Section~\ref{local intersection subsection}. This part mainly involves linear algebra.

Recall that $M_v=W_v\cap\Lambda_v=\mathcal O_{E_v}w_v$, $\Lambda_v^\perp=W_v^\perp\cap\Lambda_v$, and $r_v=v(q(w_v))$. Orthogonal projection onto $W_v$ is the map
\begin{equation*}
    \operatorname{pr}_{W_v}(y)=q(w_v)^{-1}\langle y,w_v\rangle w_v.
\end{equation*}
We first apply the local lattice classification in \cite[Section~2.1]{Guo3} to the primitive vector $w_v$.

\begin{lemma}\label{all place line types}
\begin{enumerate}[(1)]
\item At a split or inert place, $r_v\ge0$ determines the $U_v$-orbit of $M_v$, and $M_v$ lies in an orthogonal self-dual summand of rank two. At an inert place, the parity of $r_v$ determines the norm class of $W_v$.
\item At a ramified place with $r_v\ge1$, $M_v$ lies in an orthogonal $\varpi_{E_v}$-modular hyperbolic plane. Its norm class is determined by $q(w_v)$.
\item At a ramified place with $r_v=0$, $M_v$ is self-dual and $\Lambda_v=M_v\perp\Lambda_v^\perp$, where $\Lambda_v^\perp$ is $\varpi_{E_v}$-modular. This case occurs only for even $n$.
\end{enumerate}
Consequently, in these cases,
\begin{equation*}
    \operatorname{pr}_{W_v}\Lambda_v=
    \begin{cases}
        \varpi_v^{-r_v}M_v,&v\text{ split or inert},\\
        \varpi_{E_v}^{1-2r_v}M_v,&v\text{ ramified},\ r_v\ge1,\\
        M_v,&v\text{ ramified},\ r_v=0.
    \end{cases}
\end{equation*}
\end{lemma}
\begin{proof}
We apply \cite[Lemmas~2.1, 2.3, and~2.4, Corollaries~2.2 and~2.5]{Guo3} to $w_v$. At a split place, neither idempotent component of $w_v$ is divisible by $\varpi_v$ in the corresponding lattice. Thus the integers $s,t$ in \cite[Lemma~2.1]{Guo3} are both zero. At split and inert places, the norm map on units is surjective; a change of generator therefore removes the unit part of $q(w_v)$. This gives the stated orbit classification and orthogonal summands.

For the projection formula, the pairing ideal $\langle\Lambda_v,w_v\rangle$ is $\mathcal O_{E_v}$ at split and inert places, $\varpi_{E_v}\mathcal O_{E_v}$ in the ramified case $r_v\ge1$, and $\mathcal O_{E_v}$ in the ramified case $r_v=0$. Substitution in the definition of $\operatorname{pr}_{W_v}$ gives the formula for $\operatorname{pr}_{W_v}\Lambda_v$.
\end{proof}

Now, we choose bases for the summands in Lemma~\ref{all place line types}. At split and inert places, and at ramified places with $r_v\ge1$, we denote the summand of rank two by $\Lambda_{1,v}$ and its orthogonal complement in $\Lambda_v$ by $\Lambda_{1,v}^\perp$. In the bases of \cite[Section~2.1]{Guo3}, we can write
\begin{equation}\label{adapted plane bases}
    \begin{aligned}
        w_v&=e+\frac{q(w_v)}2f,&
        t_v&=e-\frac{q(w_v)}2f,
        &&v\text{ inert},\\[1mm]
        w_v&=e-\frac{q(w_v)}{2\varpi_{E_v}}f,&
        t_v&=e+\frac{q(w_v)}{2\varpi_{E_v}}f,
        &&v\text{ ramified},\ r_v\ge1.
    \end{aligned}
\end{equation}
Here $e,f$ are hyperbolic bases with $\langle e,f\rangle=1$ and $\varpi_{E_v}$, respectively. At a split place, we use the idempotent coordinates
\begin{equation*}
    w_v=\bigl((1,0),(q(w_v),1)\bigr),\qquad
    t_v=\bigl((1,-q(w_v)),(0,1)\bigr).
\end{equation*}
In all three cases, $q(t_v)=-q(w_v)$ and $\langle w_v,t_v\rangle=0$.

The correspondence in \cite[Propositions~3.3 and~3.6]{Guo2} associates a maximal quaternionic order with $\Lambda_{1,v}$. The action determined by $W_v$ embeds the quadratic algebra $E_v$ in that quaternion algebra. We will use its intersection with the maximal order. The following calculation identifies this quadratic order and gives matrices for its action. These matrices are also used in Lemma~\ref{local product stabilization} to identify the lattice in the rational Tate module of the CM formal module, and hence its lifting problem.

\begin{proposition}\label{local CM order}
The quadratic order associated with $W_v\subset E_v\Lambda_{1,v}$ is
\begin{equation*}
    \mathcal O_{F_v}+\varpi_v^{c_v}\mathcal O_{E_v},\qquad
    c_v=\begin{cases}
        r_v,&v\text{ split or inert},\\
        r_v-1,&v\text{ ramified},\ r_v\ge1.
    \end{cases}
\end{equation*}
At a ramified place with $r_v=0$, the factor $M_v$ has the full action of $\mathcal O_{E_v}$, so $c_v=0$.
\end{proposition}
\begin{proof}
The key point is to compute the integral elements of the quadratic algebra in the matrix realization of \cite[Propositions~3.3 and~3.6]{Guo2}. At an inert place, we choose $\delta\in\mathcal O_{E_v}^{\times}$ with $\bar\delta=-\delta$ and $\mathcal O_{E_v}=\mathcal O_{F_v}[\delta]$. For $A,B\in F_v$, the element $A+B\delta$ acts by
\begin{equation*}
    \begin{pmatrix}A&2B/q(w_v)\\q(w_v)\delta^2B/2&A\end{pmatrix}.
\end{equation*}
It is integral precisely when $A\in\mathcal O_{F_v}$ and $B\in\varpi_v^{r_v}\mathcal O_{F_v}$. At a split place, the two idempotent factors give
\begin{equation*}
    \{(A,B)\in\mathcal O_{F_v}^2:A-B\in\varpi_v^{r_v}\mathcal O_{F_v}\}.
\end{equation*}
This description also applies at places above $2$.

At a ramified place with $r_v\ge1$, the element $A+B\varpi_{E_v}$ acts by
\begin{equation*}
    \begin{pmatrix}A&-2B\varpi_{E_v}^2/q(w_v)\\-q(w_v)B/2&A\end{pmatrix}.
\end{equation*}
Since $v(\varpi_{E_v}^2)=1$, this matrix is integral precisely when $A\in\mathcal O_{F_v}$ and $B\in\varpi_v^{r_v-1}\mathcal O_{F_v}$. For $r_v=0$, the orthogonal decomposition in Lemma~\ref{all place line types} gives the full integral action on $M_v$.
\end{proof}

We next give the tensor decomposition used in the Whittaker calculations. Orthogonal projection identifies
\begin{equation*}
    \Lambda_v/(M_v\oplus\Lambda_v^\perp)
       \simeq\operatorname{pr}_{W_v}\Lambda_v/M_v.
\end{equation*}
For each class on the right, we choose $x_\alpha+z_\alpha\in\Lambda_v$, with $x_\alpha\in W_v^\perp$ and $z_\alpha\in W_v$. In the summand of rank two, we take $z_\alpha=\alpha w_v$ and $x_\alpha=-\alpha t_v$, where
\begin{equation}\label{coset representatives}
    \alpha\in\begin{cases}
        q(w_v)^{-1}\mathcal O_{E_v}/\mathcal O_{E_v},&v\text{ split or inert},\\
        \varpi_{E_v}q(w_v)^{-1}\mathcal O_{E_v}/\mathcal O_{E_v},&v\text{ ramified},\ r_v\ge1.
    \end{cases}
\end{equation}
For a nonzero class $\alpha$, let $k\ge1$ be the smallest integer such that $\varpi_v^k\alpha=0$ in the quotient. At a ramified place, we use $\varpi_{E_v}$ instead of $\varpi_v$. At split places this convention is applied separately to the two coordinates.

\begin{proposition}\label{lattice coset decomposition}
The lattice function has the decomposition
\begin{equation}\label{all place coset decomposition}
    1_{\Lambda_v}(x+z)=\sum_{\alpha\in\operatorname{pr}_{W_v}\Lambda_v/M_v}
       1_{x_\alpha+\Lambda_v^\perp}(x)1_{z_\alpha+M_v}(z).
\end{equation}
On $\Lambda_{1,v}$ this is
\begin{equation}\label{explicit plane cosets}
    1_{\Lambda_{1,v}}(x+z)=\sum_\alpha
       1_{-\alpha t_v+\mathcal O_{E_v}t_v}(x)
       1_{\alpha w_v+\mathcal O_{E_v}w_v}(z),
\end{equation}
with $\alpha$ as in \eqref{coset representatives}. The classes have the following descriptions.
\begin{enumerate}[(1)]
\item At an inert place, there are $N_v^{2r_v}$ classes. For $1\le k\le r_v$, the $N_v^{2k}-N_v^{2k-2}$ classes of denominator $\varpi_v^k$ have norm valuation $r_v-2k$.
\item At a ramified place with $r_v\ge1$, there are $N_v^{2r_v-1}$ classes. For $1\le k\le2r_v-1$, the $N_v^k-N_v^{k-1}$ classes of denominator $\varpi_{E_v}^k$ have norm valuation $r_v-k$. For $r_v=0$, the decomposition has one term.
\item At a split place, the quotient is $(q(w_v)^{-1}\mathcal O_{F_v}/\mathcal O_{F_v})^2$, of order $N_v^{2r_v}$. In either coordinate there are $N_v^k-N_v^{k-1}$ classes of denominator $\varpi_v^k$. Thus there are that many classes with this coordinate nonzero and the other zero, and
\begin{equation*}
    (N_v^{k_1}-N_v^{k_1-1})(N_v^{k_2}-N_v^{k_2-1})
\end{equation*}
classes with both coordinates nonzero of denominators $\varpi_v^{k_1}$ and $\varpi_v^{k_2}$.
\end{enumerate}
\end{proposition}
\begin{proof}
For each chosen vector $x_\alpha+z_\alpha\in\Lambda_v$, the vectors of $\Lambda_v$ whose $W_v$-component lies in $z_\alpha+M_v$ form
\begin{equation*}
    (x_\alpha+\Lambda_v^\perp)+(z_\alpha+M_v).
\end{equation*}
Indeed, subtracting $x_\alpha+z_\alpha$ and the resulting component in $M_v$ leaves a vector of $\Lambda_v^\perp$. These disjoint subsets give \eqref{all place coset decomposition}. The bases in \eqref{adapted plane bases} and the split coordinates show that $\alpha w_v+\beta t_v\in\Lambda_{1,v}$ precisely when $\alpha+\beta\in\mathcal O_{E_v}$ and $\alpha$ belongs to the numerator of the corresponding quotient in \eqref{coset representatives}. This gives \eqref{explicit plane cosets}.

The cardinalities follow from these quotients. At inert places, multiplication by $\varpi_v^{-k}$ changes the norm valuation by $-2k$. At ramified places, multiplication by $\varpi_{E_v}^{-k}$ changes it by $-k$. At split places the two coordinates vary independently. These observations give all three descriptions.
\end{proof}

\subsection{Local Whittaker coefficients}\label{local Whittaker subsection}

We now compute the two finite local functions in Theorem~\ref{derivative decomposition theorem}. By \eqref{c definition}, the function $c_{\Phi_v}$ is determined by the zero Whittaker coefficient. By \eqref{local k definition}, the function $k_{\Phi_v}$ is determined by nonzero coefficients in the opposite norm class. We calculate these in Propositions~\ref{zero Whittaker coefficients} and~\ref{nonzero Whittaker coefficients}, and then apply the formulas to the chosen $\Phi_v$ in Proposition~\ref{Whittaker chosen data}. The comparison with $m^0_{\Phi_v}$ and $m_{\Phi_v}$ in the height series will be made in Section~\ref{local intersection subsection}.

We first work with an additive character of conductor $\mathcal O_{F_v}$, i.e., the case when $v(d_v)=0$. At the end of the subsection, we pass to the character fixed in Section~\ref{notation subsection} when $v(d_v)>0$. We use the cosets $z_\alpha+M_v$ of Proposition~\ref{lattice coset decomposition}; the Whittaker functional acts only on this one-dimensional factor.

Recall the local Whittaker integrals of \cite[Section~6.1]{YZZ} and the congruence calculation of \cite[Section~7.3]{YZ}. For $a\ne0$, this gives
\begin{equation}\label{Whittaker ball sum}
    W_{a,v}^{\circ}(s,1,\phi_{2,v})=(1-N_v^{-s})\sum_{j\ge0}N_v^{j(1-s)}\int_{q(z)-a\in\varpi_v^j\mathcal O_{F_v}}\phi_{2,v}(z)\,\dd z.
\end{equation}
For $a=0$, the right side gives $\gamma_v^{-1}W_{0,v}$ before the additional normalization in Section~\ref{mixed subsection}. To obtain the formula, we integrate over $\mathcal O_{F_v}$ and the subsets $v(b)=-j$, $j\ge1$. The latter have integral
\begin{equation*}
    \int_{v(b)=-j}\psi_v(b(q(z)-a))\,\dd b
      =N_v^j1_{\varpi_v^j\mathcal O_{F_v}}(q(z)-a)
        -N_v^{j-1}1_{\varpi_v^{j-1}\mathcal O_{F_v}}(q(z)-a).
\end{equation*}
For $v(b)=-j$, we have $\delta_v(wn(b))^s=N_v^{-js}$. Combining this with the preceding equality and summing over $j$, we obtain \eqref{Whittaker ball sum}.

\subsubsection*{The zero Whittaker coefficient}

The zero coefficients determine $c_{\Phi_v}$ through \eqref{c definition}. The function $a_v$ is the one in Theorem~\ref{CM height theorem}. In the formulas below, $k$ denotes the denominator exponent of a nonzero coset, as defined before Proposition~\ref{lattice coset decomposition}.

\begin{proposition}\label{zero Whittaker coefficients}
For the zero coset at every finite place,
\begin{equation}\label{zero Whittaker value derivative}
    W_{0,v}^\circ(0,1,1_{M_v})=1,\qquad
    (W_{0,v}^\circ)'(0,1,1_{M_v})=-a_v(r_v)\log N_v.
\end{equation}
For a nonzero coset, its normalized value at zero is $0$, and its derivative is as follows.
\begin{enumerate}[(1)]
\item At a split place, if exactly one coordinate is nonzero with denominator $\varpi_v^k$, then
\begin{equation*}
    (W_{0,v}^\circ)'(0,1,1_{z_\alpha+M_v})
       =\frac{N_v^{-k}}{1-N_v^{-1}}\log N_v.
\end{equation*}
If both coordinates are nonzero, the derivative is zero.
\item At an inert place,
\begin{equation}\label{inert coset zero derivative}
    (W_{0,v}^\circ)'(0,1,1_{z_\alpha+M_v})=\begin{cases}\displaystyle\frac{2N_v^{1-r_v}(N_v^{r_v-2k+1}-1)}{N_v^2-1}\log N_v,&2k\le r_v,\\[2mm]0,&2k>r_v.\end{cases}
\end{equation}
\item At a ramified place,
\begin{equation}\label{ramified coset zero}
    (W_{0,v}^{\circ})'(0,1,1_{z_\alpha+M_v})=\begin{cases}\displaystyle N_v^{-r_v}\frac{N_v^{r_v-k+1}-1}{N_v-1}\log N_v,&k\le r_v,\\[2mm]0,&k>r_v.\end{cases}
\end{equation}
\end{enumerate}
\end{proposition}
\begin{proof}
Our strategy is to compute the volumes in \eqref{Whittaker ball sum} and apply the normalization at zero. We first treat $M_v$ and then the nonzero cosets.

\smallskip
\noindent\textbf{Step 1.} For $j\ge0$, the volumes of $\{z\in M_v:q(z)\in\varpi_v^j\mathcal O_{F_v}\}$ are
\begin{equation*}
    \begin{cases}
        N_v^{-r_v},&j\le r_v,\quad v\text{ split or inert},\\
        N_v^{-j}\bigl(1+(1-N_v^{-1})(j-r_v)\bigr),&j\ge r_v,\quad v\text{ split},\\
        N_v^{-r_v-2\lceil(j-r_v)/2\rceil},&j>r_v,\quad v\text{ inert},\\
        N_v^{-\max\{j,r_v\}-1/2},&v\text{ ramified}.
    \end{cases}
\end{equation*}
At split places this follows from $q(z_1,z_2)=q(w_v)z_1z_2$. At inert and ramified places, the valuations of norms are respectively twice the valuation in $E_v$ and the valuation in $E_v$.

Substitution of these volumes in \eqref{Whittaker ball sum} reduces the coefficient to geometric sums. The additional normalization of $W_{0,v}^\circ$ is $(1-N_v^{-s})/(1-N_v^{-s-1})$ at split places, $(1+N_v^{-s})/(1+N_v^{-s-1})$ at inert places, and $N_v^{1/2}$ at ramified places. After taking the geometric sums, we obtain, respectively,
\begin{equation}\label{split centered family}
    W_{0,v}^{\circ}(s,1,1_{M_v})=N_v^{-r_v s}+\frac{(1-N_v^{-s})^2N_v^{-r_v}}{1-N_v^{-s-1}}\sum_{j=0}^{r_v-1}N_v^{j(1-s)}.
\end{equation}
\begin{equation*}
    W_{0,v}^\circ(s,1,1_{M_v})=\frac{N_v^{-r_v}(1+N_v^{-s})\left(1+(1-N_v^{-1})\sum_{j=1}^{r_v}N_v^{j(1-s)}\right)-(1-N_v^{-1})N_v^{-(r_v+1)s}}{1+N_v^{-s-1}}.
\end{equation*}
\begin{equation}\label{ramified centered family}
    W_{0,v}^{\circ}(s,1,1_{M_v})=N_v^{-r_v s}+(1-N_v^{-s})N_v^{-r_v}\sum_{j=0}^{r_v-1}N_v^{j(1-s)}.
\end{equation}
Evaluation and differentiation at $s=0$ give \eqref{zero Whittaker value derivative}.

\smallskip
\noindent\textbf{Step 2.} Next, we evaluate the nonzero cosets. At a split place with one nonzero coordinate of denominator $\varpi_v^k$, the norm map sends the measure on the coset to a constant multiple of Haar measure on $\varpi_v^{r_v-k}\mathcal O_{F_v}$. Its total volume is $N_v^{-r_v}$, and for $j\ge r_v-k$ the inverse image of $\varpi_v^j\mathcal O_{F_v}$ has volume $N_v^{-k-j}$. The infinite part of \eqref{Whittaker ball sum} is therefore $N_v^{-k-(r_v-k)s}$. After the normalization of $W_{0,v}^\circ$, differentiation gives the stated formula for this coset; the finite part has two factors $1-N_v^{-s}$ and has zero derivative. If both coordinates are nonzero, all norms have valuation $r_v-k_1-k_2$. Only finitely many terms remain in \eqref{Whittaker ball sum}, so the same two factors give derivative zero.

At an inert place, all vectors in a nonzero coset of denominator $\varpi_v^k$ have norm valuation $r_v-2k$. The total volume is $N_v^{-r_v}$. Thus the sum in \eqref{Whittaker ball sum} stops at $j=r_v-2k$. After taking the derivative and multiplying by the normalization factor $2/(1+N_v^{-1})$, we obtain \eqref{inert coset zero derivative}. The sum is empty when $2k>r_v$.

At a ramified place, the norm valuation on the coset is $r_v-k$, and its volume is $N_v^{-r_v-1/2}$. The same calculation, with normalization factor $N_v^{1/2}$, gives
\begin{equation*}
    N_v^{-r_v}\sum_{j=0}^{r_v-k}N_v^j\log N_v
\end{equation*}
when $k\le r_v$, and zero otherwise. This is \eqref{ramified coset zero}. Note that in all three cases, the normalized value at $s=0$ is zero.
\end{proof}

\subsubsection*{The nonzero Whittaker coefficients}

Now, we compute the coefficients which determine $k_{\Phi_v}$. Here $v$ is nonsplit and $a=q(y_2)$ with $0\ne y_2\in W(v)_v$, so $\eta_v(a/q(w_v))=-1$. This condition makes the local Whittaker value at $s=0$ vanish and allows us to compute its derivative by a finite sum.

\begin{proposition}\label{nonzero Whittaker coefficients}
For $a\ne0$ with $\eta_v(a/q(w_v))=-1$, we have
\begin{equation*}
    W_{a,v}^\circ(0,1,1_{z_\alpha+M_v})=0.
\end{equation*}
The derivatives are given by the following formulas.
\begin{enumerate}[(1)]
\item At an inert place, the coefficient on the zero coset is
\begin{equation}\label{inert opposite derivative}
    \frac{2(W_{a,v}^\circ)'(0,1,1_{M_v})}{1+N_v^{-1}}=\begin{cases}\displaystyle\frac{2(N_v^{v(a)+1}-1)}{N_v^{r_v-1}(N_v^2-1)}\log N_v,&0\le v(a)<r_v,\\[2mm](v(a)+1-a_v(r_v))\log N_v,&v(a)\ge r_v.\end{cases}
\end{equation}
For a nonzero coset of denominator $\varpi_v^k$, it is
\begin{equation*}
    \frac{2(W_{a,v}^\circ)'(0,1,1_{z_\alpha+M_v})}{1+N_v^{-1}}
    =\begin{cases}
        \displaystyle\frac{2N_v^{1-r_v}(N_v^{\min\{v(a),r_v-2k\}+1}-1)}{N_v^2-1}\log N_v,
             &v(a)\ge0,\ 2k\le r_v,\\[2mm]
        0,&\text{otherwise}.
    \end{cases}
\end{equation*}
\item At a ramified place, the coefficient on the zero coset is
\begin{equation}\label{ramified center nonzero derivative}
    N_v^{1/2}(W_{a,v}^\circ)'(0,1,1_{M_v})
      =\begin{cases}
        \displaystyle N_v^{-r_v}\frac{N_v^{v(a)+1}-1}{N_v-1}\log N_v,&0\le v(a)<r_v,\\[2mm]
        \displaystyle\left(v(a)+1-r_v+\frac{1-N_v^{-r_v}}{N_v-1}\right)\log N_v,&v(a)\ge r_v.
      \end{cases}
\end{equation}
For a nonzero coset of denominator $\varpi_{E_v}^k$, it is
\begin{equation*}
    N_v^{1/2}(W_{a,v}^\circ)'(0,1,1_{z_\alpha+M_v})
      =\begin{cases}
        \displaystyle N_v^{-r_v}\frac{N_v^{\min\{r_v-k,v(a)\}+1}-1}{N_v-1}\log N_v,
             &\min\{r_v-k,v(a)\}\ge0,\\[2mm]
        0,&\min\{r_v-k,v(a)\}<0.
      \end{cases}
\end{equation*}
\end{enumerate}
For $v(a)<0$, all these Whittaker functions are identically zero.
\end{proposition}
\begin{proof}
Since $a$ is in the opposite norm class, we have $v(q(z)-a)=\min\{v(q(z)),v(a)\}$. The result follows from \eqref{Whittaker ball sum}. Indeed, at an inert place, $v(a)$ and $r_v$ have opposite parity. At a ramified place, if $q(z)$ and $a$ have equal valuation, their quotient has nonsquare residue. These are the two cases of \cite[Lemma~7.5]{YZ}. Thus the sum stops at $j=v(a)$ and its derivative is $\log N_v$ times the sum of $N_v^j$ multiplied by the volumes used in the preceding proof.

On the zero coset, we substitute the volumes from Step~1 of Proposition~\ref{zero Whittaker coefficients}. After taking the finite geometric sums, we obtain \eqref{inert opposite derivative} and \eqref{ramified center nonzero derivative}. On a nonzero coset, the norm has constant valuation $r_v-2k$ at an inert place and $r_v-k$ at a ramified place. The sum therefore stops at the minimum of that valuation and $v(a)$. Its volume is $N_v^{-r_v}$ or $N_v^{-r_v-1/2}$, respectively. This gives the stated formulas for the nonzero cosets. A negative minimum gives an empty sum.

For later use at ramified places, the same calculation before differentiation gives
\begin{equation}\label{ramified centered nonzero}
    W_{a,v}^{\circ}(s,1,1_{M_v})=N_v^{-1/2}\left(N_v^{-r_v s}+\eta_v(a/q(w_v))N_v^{-(v(a)+1)s}+(1-N_v^{-s})N_v^{-r_v}\sum_{j=0}^{r_v-1}N_v^{j(1-s)}\right)
\end{equation}
when $v(a)\ge r_v$. For $0\le v(a)<r_v$, it gives
\begin{equation*}
    W_{a,v}^\circ(s,1,1_{M_v})
       =N_v^{-r_v-1/2}(1-N_v^{-s})\sum_{j=0}^{v(a)}N_v^{j(1-s)}.
\end{equation*}
For $r_v\le v(a) <j$, the congruence has volume $2N_v^{-j-1/2}$ if $\eta_v(a/q(w_v))=1$, and zero otherwise. Thus \eqref{ramified centered nonzero} holds in both cases. This finishes the proof.
\end{proof}

\subsubsection*{The local terms for the chosen Schwartz function}

The preceding coefficients now determine the local terms for the functions of Section~\ref{split data subsection}. At a place where $\Phi_v=1_{\Lambda_v}$, Proposition~\ref{lattice coset decomposition} and \eqref{c definition} give
\begin{equation}\label{c coset expansion}
    c_{\Phi_v}(1,y_1)=\sum_\alpha
       1_{x_\alpha+\Lambda_v^\perp}(y_1)
       (W_{0,v}^\circ)'(0,1,1_{z_\alpha+M_v}).
\end{equation}
At a nonsplit place, the same decomposition and \eqref{local k definition} give
\begin{equation}\label{k coset expansion}
    2k_{\Phi_v}(1,y_1+y_2)=\frac{2L_v(1,\eta_v)}{\vol(E_v^1)}
       \sum_\alpha1_{x_\alpha+\Lambda_v^\perp}(y_1)
       (W_{q(y_2),v}^\circ)'(0,1,1_{z_\alpha+M_v}).
\end{equation}
The factor in this formula is $2/(1+N_v^{-1})$ at an inert place with $v(d_v)=0$, and $N_v^{1/2}$ at a ramified place. Thus Proposition~\ref{nonzero Whittaker coefficients} already gives the coefficients in the normalization used for the height comparison.

\begin{proposition}\label{Whittaker chosen data}
The function $c_{\Phi_v}(g_v,\cdot)$ is Schwartz on $W_v^\perp$ and is zero at almost every place. For the functions chosen in Section~\ref{split data subsection}, we have the following additional formulas.
\begin{enumerate}[(1)]
\item At $v\in S_2$,
\begin{equation}\label{split zero derivative}
    c_{\Phi_v}(1,y_1)=-2\log N_v\,1_{\varpi_v\Lambda_v^\perp}(y_1)1_{\varpi_v^2\mathcal O_{F_v}^\times}(q(y_1)).
\end{equation}
\item At $v\in S_0(W)$, $c_{\Phi_v}(1,\cdot)=0$, and the coefficients of $2k_{\Phi_v}$ are those of \eqref{ramified center nonzero derivative} at $r_v=0$, multiplied by $1_{(\Lambda_v^\perp)^\vee}(y_1)$.
\item For the character fixed in Section~\ref{notation subsection}, if $v(d_v)>0$, the zero coefficient on $M_v$ is
\begin{equation*}
    (W_{0,v}^\circ)'(0,1,1_{M_v})
       =-a_v(r_v+v(d_v))\log N_v.
\end{equation*}
At inert places, the formulas for nonzero cosets in Proposition~\ref{zero Whittaker coefficients} use $r_v+v(d_v)$ in place of $r_v$. The formulas for $2k_{\Phi_v}$ use $r_v+v(d_v)$ and $v(a)+v(d_v)$ in place of $r_v$ and $v(a)$. At split places, the zero derivatives on nonzero cosets are unchanged.
\end{enumerate}
\end{proposition}
\begin{proof}
The Schwartz property follows from the finite tensor decomposition in \eqref{c definition}. At a standard place with $r_v=v(d_v)=0$, \cite[Lemma~4.8]{Guo1} gives
\begin{equation*}
    W_{0,v}^\circ(s,g_v,1_{M_v})
       =\delta_v(g_v)^{-s}r_{\BW_v}(g_v)1_{M_v}(0).
\end{equation*}
Differentiation cancels the second term in \eqref{c definition}, which proves vanishing at almost every place.

For (1), we identify $M_v$ with $\mathcal O_{F_v}^2$, with norm $z_1z_2$. The generators of $M_v$ and $\varpi_vM_v$ have norm valuations $0$ and $2$. Formula~\eqref{split centered family} therefore gives zero derivatives $0$ and $-2\log N_v$, respectively. A function whose norms are bounded away from zero contributes a finite sum in \eqref{Whittaker ball sum} for $a=0$. After the normalization of $W_{0,v}^\circ$, it has two factors $1-N_v^{-s}$, so its derivative is zero. For fixed $y_1$, the difference
\begin{equation*}
    1_{\varpi_v^2\mathcal O_{F_v}^{\times}}(q(y_1)+q(z))
       -1_{\varpi_v^2\mathcal O_{F_v}^{\times}}(q(y_1))
\end{equation*}
is supported where $q(z)$ is bounded away from zero. The norm condition in \eqref{degenerate phi} therefore multiplies the two derivatives by $1_{\varpi_v^2\mathcal O_{F_v}^{\times}}(q(y_1))$. This proves \eqref{split zero derivative}.

For (2), the one-dimensional factor is $1_{M_v}$ with $r_v=0$. Formulas~\eqref{zero Whittaker value derivative} and~\eqref{ramified center nonzero derivative} give the assertions.

For (3), we compare the fixed character with a character of conductor $\mathcal O_{F_v}$ by multiplication by a generator of $d_v$. The congruences in \eqref{Whittaker ball sum} then involve $d_vq$ and $d_va$. The self-dual measure on $F_v$ contributes $N_v^{-v(d_v)/2}$, which cancels $|d_v|_v^{-1/2}$ in the normalization of $W_{0,v}^\circ$. This gives the stated zero Whittaker derivative and the stated shift of the nonzero cosets. In \eqref{k coset expansion}, the same measure factor cancels $\vol(E_v^1)^{-1}$. Hence the nonzero coefficients have the stated shifts. 
\end{proof}

\section{Height series}\label{height series section}

In this section, we define the height series $\mathcal H(g,\Phi)$ in terms of arithmetic intersections and compute its local contributions. Theorem~\ref{height decomposition theorem} gives an explicit decomposition, and separates the proper intersections from the degrees of divisor line bundles restricted to $\mathcal P$. We compute the former in Section~\ref{local intersection subsection} and the latter in Section~\ref{adjunction subsection}. These are the arithmetic terms that will be compared with the derivative series.

To ensure that both the Shimura variety and its integral model are proper, we assume throughout this section that $F\neq\mathbb{Q}$. However, many of the local computations in this section remain valid when $F=\mathbb{Q}$.

\subsection{Integral models and arithmetic generating series}\label{integral model subsection}

First, we recall the integral model, Hodge line bundle, and CM cycle. We use the arithmetic intersection theory of Gillet--Soul\'e \cite{GS}. For admissible extensions, we follow the pairings in \cite{YZZ} and their higher-dimensional formulation in \cite[Appendix~A]{Qiu}, with the normalization of \cite[Sections~2.2 and~4.1]{Guo2}.

For $V$ and $U$ fixed in Section~\ref{notation subsection}, the associated Shimura variety has complex points
\begin{equation*}
    X_U(\C)=\U(V)(F)\backslash\bigl(D\times\U(V)(\Af)\bigr)/U,
    \qquad D=\{z\in\mathbf P(V_\iota):q(z)<0\}.
\end{equation*}
Since $F\ne\Q$, the space $V$ is definite at another real place and is anisotropic over $F$. Thus $X_U$ is proper. The construction of Rapoport--Smithling--Zhang \cite{RSZ3}, in the form of \cite[Theorem~4.4.4 and Corollary~4.4.8]{Qiu}, gives the integral model after a common reflex-field extension. We use the normalization in \cite[Section~2.2]{Guo2} and denote the resulting model by $\mathcal X_U$.

The tautological line bundle on $D$, with metric $\|s_z\|^2=-q(s_z)$, descends to the Hodge line bundle $L$ on $X_U$.We use the integral extension of the Hodge line bundle on the RSZ model, as defined in \cite[Section~2.3]{Guo2}. Note that in the local computations below, a chart means an \'etale neighborhood on this model, or its formal completion when considering formal groups. On such a chart $\mathcal X$, let $\mathcal A_0$ be the auxiliary CM abelian scheme and $\mathcal A$ the universal abelian scheme. At a place unramified in $E/F$, the signature determines the embedding set $\mathfrak A_v$ and its Eisenstein ideal $J_{\mathfrak A_v}$, as in \cite[Section~2.3]{Guo2}. The inverse Hodge bundle is
\begin{equation*}
    \mathcal L^{-1}=\operatorname{Hom}_{\mathcal O_E\otimes\mathcal O_{\mathcal X}}
       \bigl(\Fil\mathcal A_0,(J_{\mathfrak A_v}\Lie\mathcal A)_1\bigr).
\end{equation*}
The subscript $1$ denotes the line bundle selected by the signature. The relative modification in \cite[Proposition~7.8]{LMZ} preserves this quotient and its generic identification, including the case when $F_v/\Q_p$ is ramified.

At a place $v$ ramified in $E/F$, we also denote by $\varpi_{E_v}$ its image in $\mathcal O_{\mathcal X}$. Then we define
\begin{equation*}
    \mathcal I_{\mathcal A}=\operatorname{im}\bigl(
        \iota_{\mathcal A}(\varpi_{E_v})-\varpi_{E_v}:
        \Lie\mathcal A\longrightarrow\Lie\mathcal A\bigr).
\end{equation*}
We define $\mathcal L^{-1}=\operatorname{Hom}(\Fil\mathcal A_0,\mathcal I_{\mathcal A})$ for a $\varpi_{E_v}$-modular lattice, and use $\mathcal I_{\mathcal A}^{\vee\vee}$ in the almost $\varpi_{E_v}$-modular case. These are the extensions in \cite[Section~2.3]{Guo2}. With the tautological metric, we denote the metrized line bundle by $\LL$.

Next, we define the normalized CM cycle used in the pairings. Let $f_W:P_W\to X_U$ be the finite morphism from the zero-dimensional Shimura variety of $\U(W)$ at level $U\cap\U(W)$. Its pushforward includes the stabilizer multiplicities. We put
\begin{equation}\label{CM cycle normalization}
    P=\frac{(f_W)_*[P_W]}{\deg((f_W)_*[P_W])}.
\end{equation}
Its normalized arithmetic closure on $\mathcal X_U$ is denoted by $\mathcal P$. On the RSZ integral model, the normalization divides by the degrees of the auxiliary toric factor and the reflex-field extension, as in \cite[Section~2.2]{Guo2}.

For an arithmetic divisor $\widehat{\mathcal D}$ meeting $\mathcal P$ properly, the arithmetic intersection is
\begin{equation}\label{arithmetic normalization}
    \widehat{\mathcal D}\cdot\mathcal P=\frac{1}{\delta_Z[\widetilde E:E]}\left(\sum_{\widetilde\nu\nmid\infty}\Int_{\widetilde\nu}(\mathcal D,\mathcal P)\log N\widetilde\nu+\sum_{\widetilde\sigma:\widetilde E\hookrightarrow\C}g_{\mathcal D,\widetilde\sigma}(\mathcal P)\right).
\end{equation}
Here $\widetilde E$ is a common reflex field, $\delta_Z$ is the degree of the auxiliary toric factor, and $\Int_{\widetilde\nu}$ is the Euler characteristic of the derived intersection on the relative local model. Both sums include the multiplicities in \eqref{CM cycle normalization}. Both places of $E$ above a split $v$ contribute to the finite sum.

Meanwhile, if the generic fiber of an arithmetic divisor contains a CM point, its intersection with the corresponding component of $\mathcal P$ is given by the arithmetic degree of the associated metrized divisor line bundle restricted to that component. We use this normalization and the projection formula for finite pullback throughout the local computations.

Next we briefly recall the arithmetic generating series of special divisors. For $t\in F_{>0}$, let $V_{f,t}=\{y\in V(\A_{E,f}):q(y)=t\}$, where $\A_{E,f}$ denotes the finite adeles of $E$. Let $Z(y)_U$ be the special divisor attached to $y^\perp$, with its finite map to $X_U$ and stabilizer multiplicities. With the standard archimedean Whittaker factor understood, we put
\begin{equation*}
    Z_t(g,\Phi)=\sum_{y\in U\backslash V_{f,t}}r(g)\Phi(y)Z(y)_U,
    \qquad Z_*(g,\Phi)=\sum_{t\in F_{>0}}Z_t(g,\Phi).
\end{equation*}
The finite functions are $U$-invariant. For their translates, we use a common smaller level and the projection formula.

Let $\widehat{\mathcal Z}_t(g,\Phi)$ be the normalized $\LL$-admissible extension. More precisely, its Green function has mean zero on each connected component. At the smooth maximal level, its finite representative is the horizontal closure by \cite[Lemma~A.2.4]{Qiu}. The lifting description of these closures is given by \cite[Section~6.1.2 and Proposition~6.1.5]{Qiu}. For the dual lattice at $S_0(W)$, Lemma~\ref{unit horizontal divisor} identifies the closure with the lifting divisor of $\lambda x$, where $\lambda:\mathcal A\to\mathcal A^\vee$ is the polarization. At $v\in S_2$, Proposition~\ref{split arithmetic correction} computes the correction arising from requiring $x/\varpi_v$ to lift instead of $x$.

We put
\begin{equation*}
    \deg_L X_U=c_1(L)^n\cdot[X_U],\qquad
    \widehat\xi=\frac{\widehat c_1(\LL)^n}{\deg_L X_U},\qquad
    \widehat{\mathcal Z}_*(g,\Phi)=\sum_{t\in F_{>0}}\widehat{\mathcal Z}_t(g,\Phi).
\end{equation*}
For each connected component $X_i$, the normalized cycle satisfies
\begin{equation*}
    \deg(P|_{X_i})=\frac{\deg_L X_i}{\deg_L X_U}.
\end{equation*}
Indeed, the determinant on the embedded $\U(W)$ is the identity. Strong approximation for $\operatorname{SU}(V)$ identifies the map on components with the quotient from the torus level $U\cap\U(W)$ to $\det U$. Its fibers have equal weighted cardinality. The invariant measure also gives the $L$-degrees, as in \cite[equation~(2.3.5)]{Guo2}. Thus $P-L^n/\deg_L X_U$ has degree zero on each component.

The height series is the pairing
\begin{equation}\label{height series definition}
    \mathcal H(g,\Phi)=\left(\widehat{\mathcal Z}_*(g,\Phi)-\frac{\deg_L Z_*(g,\Phi)}{\deg_L X_U}\LL\right)\cdot(\mathcal P-\widehat\xi).
\end{equation}
The following proposition removes the degree and Hodge terms for our choice of $\Phi$. It also gives the modularity needed in the final comparison.

\begin{proposition}\label{cuspidal height pairing}
The height series is a holomorphic cusp form of the weight in Section~\ref{mixed subsection}. For the datum of Section~\ref{split data subsection} and $g_f\in K_f$,
\begin{equation*}
    \deg_L Z_*(g,\Phi)=0,\qquad
    \widehat{\mathcal Z}_*(g,\Phi)\cdot\widehat\xi=0.
\end{equation*}
Consequently,
\begin{equation}\label{height P minus xi}
    \mathcal H(g,\Phi)=\widehat{\mathcal Z}_*(g,\Phi)\cdot\mathcal P.
\end{equation}
\end{proposition}
\begin{proof}
Modularity of the geometric generating series gives the first assertion. The two places in $S_2$ then remove the degree and Hodge pairings, by the argument of \cite[Proposition~7.5]{YZZ} and \cite[Section~8.1]{YZ}.

\smallskip
\noindent\textbf{Step 1.} Liu's modularity theorem \cite[Theorem~3.5]{Liu1} applies to the pairing with the cycle of degree zero, using the admissible extensions in \cite[Appendix~A.3--A.4]{Qiu}. Two admissible lifts differ by arithmetic classes from the base on each component. Their pairings with $\mathcal P-\widehat\xi$ vanish by the formula for the component degrees above. Hence the height pairing depends only on the generic divisor class and is modular.

In the full generating series, the constant term is $-r(g)\Phi(0)\LL$, and its geometric degree is $-r(g)\Phi(0)\deg_L X_U$. They cancel in the degree correction in \eqref{height series definition}. The same calculation after a finite translate gives vanishing at every cusp. The geometric Siegel--Weil formula expresses the positive degree series as a multiple of $E_{\BV,*}(0,g,\Phi)$, which is zero by Proposition~\ref{two place proposition}.

\smallskip
\noindent\textbf{Step 2.} We fix a positive Fourier index $a$ for which the divisor is nonempty. The Hasse principle gives $y_0\in V(E)$ with $q(y_0)=a$. By Witt's theorem, all divisors of this norm are finite images of Shimura varieties for $\U(y_0^\perp)$ at suitable finite levels. On a fixed such variety, choose the restriction of $\LL$ under the embedding defined by $y_0$, and pull it back under the finite maps from smaller levels. The projection formula makes its normalized arithmetic self-intersection independent of the level.

Now we compare the restrictions under the other embeddings with this fixed metrized line bundle. For a fixed $v$, two vectors in the same $U_v$-orbit can first be made equal by an element of $U_v$. Their isometries can then be chosen with equal $v$-components. On a common finite cover, the resulting correspondence has identity $v$-component and preserves the relative $v$-divisible group and the integral Hodge bundle. Thus the normalized contribution from the change of the $v$-adic metric depends only on the $U_v$-orbit of $y_v$. This applies to each term in the difference of the two $n$-fold arithmetic intersections. At infinity there is no difference: the embeddings are induced by Hermitian isometries, and both metrics are the tautological metric $-q$. For $n=1$, the argument is the projection formula for line bundle degrees on the complete CM cycle.

It remains to sum over the finite orbits defining the coefficient at $a$. The common normalized height is multiplied by the geometric degree, which is zero. For a change of metric at $v$, the local Siegel--Weil formula factors the sum over the other places as
\begin{equation*}
    \prod_{w\ne v}W_{a,w}(0,g_w,\Phi_w).
\end{equation*}
These are the Whittaker functions for $\BV$. This product has a zero factor at a place of $S_2$ different from $v$. Hence every local contribution vanishes as well. The normalized Green functions have zero integral against $c_1(L)^n$, so this proves the vanishing of the Hodge pairing. Equation~\eqref{height P minus xi} follows.
\end{proof}

\subsection{Decomposition of the height series}\label{standard finite subsection}

By Proposition~\ref{cuspidal height pairing}, the height series is the pairing of the arithmetic divisor series with $\mathcal P$. To decompose this pairing, we distinguish whether a generic CM point lies on the special divisor. If it does not, the contribution consists of proper finite intersections and Green values. If it does, the pairing is the arithmetic degree of the divisor line bundle restricted to that CM component. We first define the proper terms, at infinity and then at finite places, and afterwards define the restricted degrees. Note that all terms use \eqref{arithmetic normalization}.

\subsubsection*{Archimedean proper intersections}

Let us first describe the Green value of a special divisor which avoids the CM point. We use the function of Oda--Tsuzuki \cite{OT}, in the unitary formulation of \cite[Section~4.2.2]{Qiu}. With the normalization in \cite[Section~5.2]{Guo2}, it is, for $\rho>0$,
\begin{equation}\label{Q kernel}
    Q_{n,s}(1+\rho)=\frac12\int_0^1\frac{(1-t)^s t^{n+s-1}}{(\rho+t)^{n+s}}\,\dd t.
\end{equation}
At a real place $v$, write $V(v)=W^\perp\oplus W(v)$. For $y=y_1+y_2$ of positive norm with $y_2\ne0$, the argument at the CM point is $\rho=-q_v(y_2)/q_v(y)>0$. For each fixed positive norm, we sum \eqref{Q kernel} over these vectors with coefficient $r(g,\tau)\Phi(y)$. We use the meromorphic continuation of \cite[Section~5.2]{Guo2} and take the Laurent coefficient of degree zero. The sum of these Fourier coefficients is denoted by $\mathcal M_\Phi^{(v)}(g,\tau)$.

The condition $y_2\ne0$ omits the finitely many divisors through the CM point at each Fourier index. We treat their restricted line bundles below, using the singularity at $\rho=0$ to define the metric. The other branches of the same global divisors remain in the Green sum.

\subsubsection*{Finite proper intersections}

Generally speaking, at a finite place, the proper terms are the lengths of the local intersections.

At a finite nonsplit place, we write $y=y_1+y_2$ in $V(v)_v=W_v^\perp\perp W(v)_v$. For $y_2\ne0$, let $m_{\Phi_v}(y)\log N_v$ be the normalized proper intersection with the local horizontal Cartier divisor, with the multiplicities in \eqref{arithmetic normalization}. At $S_0(W)$ its defining homomorphism is $\lambda y$; elsewhere it is the homomorphism in the standard lifting description. For $y_1\in W_v^\perp$, we also put
\begin{equation*}
    m^0_{\Phi_v}(y_1)=\begin{cases}\text{normalized proper length},&\text{if the intersection is proper},\\0,&\text{if it lifts over the CM component}.\end{cases}
\end{equation*}
At a split place, $m^0_{\Phi_v}$ includes both places of $E$ above $v$. Both multiplicities omit $\log N_v$ and are linear in $\Phi_v$. Similar to the nonsplit case, when the homomorphism lifts over the whole CM component, we set $m^0_{\Phi_v}=0$.

For a finite nonsplit place, define the corresponding global sum by
\begin{equation*}
    \mathcal M_\Phi^{(v)}(g,\tau)
    =\sum_{y\in V(v)(E)\setminus W^\perp(E)}m_{r(\tau_v)\Phi_v}(y_v)r(g^v,\tau^v)\Phi^v(y).
\end{equation*}
The local multiplicity is computed at $g_f=1$, and the formula holds on $K_f$. At a split place the $p$-divisible group is ordinary. Its decomposition into \'etale and connected parts forces the $W_v$-component of a special homomorphism to vanish, as in \cite[Section~5.3]{Guo2}. Hence its proper finite terms are included in $m^0_{\Phi_v}$.

Now, we pass from these local multiplicities to the global divisor series by the finite maps in \cite[Section~4.3]{Guo2} and the formal uniformization in \cite[Section~6.1.2 and Proposition~6.1.5]{Qiu}. To do this, we work on the RSZ integral model at a sufficiently small level, where the special homomorphism and its level data are fixed. The local divisor is given by the horizontal Cartier closure of the corresponding divisor on the generic fiber. Since the torus fixes $W^\perp$ pointwise, the projection formula, together with the finite degrees and stabilizer weights, gives
\begin{equation}\label{perpendicular proper global sum}
    \log N_v\sum_{y_1\in W^\perp(E)\setminus\{0\}}m^0_{\Phi_v}(y_{1,v})r(g^v)\Phi^v(y_1,0).
\end{equation}
Here $g_f=1$; ordinary uniformization gives the same formula at a split place. The projection formula gives the weights of \eqref{arithmetic normalization}, as in \cite[Propositions~8.3 and~8.5]{YZ}.

To compare the finite proper terms with the derivative series, for $y_2\ne0$, we put 
\begin{equation}\label{general proper difference}
    \bar k_{\Phi_v}(1,y)=2k_{\Phi_v}(1,y)-m_{\Phi_v}(y)\log N_v.
\end{equation}
The calculations in Section~\ref{local intersection subsection} give a Schwartz extension at every finite nonsplit place. We use its restriction to $W_v^\perp$ when comparing the terms indexed by $x\in W^\perp(E)$ in Proposition~\ref{direct comparison proposition}.

\subsubsection*{Divisors through the generic CM point}

The remaining terms come from special divisors that contain a generic point of $P$. The point determined by $W$ lies on the hyperplane $\langle z,x\rangle=0$ exactly when $x\in W^\perp$. For such a nonzero $x$, the intersection is computed by restricting the metrized divisor line bundle to the CM component, rather than by a finite length.

To compute this degree, we fix the homomorphism $x$ and a CM component, and work on an \'etale neighborhood of a point on this component in the RSZ integral model, as in Section~\ref{integral model subsection}. Let $D_x$ be its horizontal lifting divisor. At $S_0(W)$ the defining homomorphism is $\lambda x$; at the other places it is $x$. At $v\in S_2$, we first consider the divisor defined by lifting $x$. Proposition~\ref{split arithmetic correction} then computes the correction when $x/\varpi_v$ is required to lift instead.

The reduced generic branch of $D_x$ through the chosen point is unique. Since the local ring of $\mathcal X_U$ is regular, its closure $C_x$ is Cartier. The remaining divisor $R_x=D_x-C_x$ meets the CM component properly. Thus the contribution of $D_x$ is
\begin{equation}\label{restricted divisor degree}
    D_x=C_x+R_x,\qquad i(x,\mathcal P)=\widehat{\deg}\bigl(\overline{\mathcal O(C_x)}|_{\mathcal P}\bigr)+(R_x,\mathcal P)_{\mathrm{fin}}.
\end{equation}
The metric on $\mathcal O(C_x)$ is obtained by subtracting the logarithmic singularity of the Green function along $C_x$, as in \cite[Lemma~5.1]{Guo2}. Both terms include the finite maps and multiplicities of \eqref{arithmetic normalization}. 

Note that $R_x$ is part of the divisor of a homomorphism that lifts over the whole CM component. Its finite intersections therefore belong to $i(x,\mathcal P)$, whereas $m^0_{\Phi_v}$ concerns homomorphisms with finite lifting locus. The following formula includes both contributions and the additional correction at $S_2$.

\begin{theorem}\label{height decomposition theorem}
For the datum of Section~\ref{split data subsection} and $g_f\in K_f$, the height series satisfies
\begin{equation}\label{height decomposition}
    \begin{aligned}
    \mathcal H(g,\Phi)={}&2\sum_{v\mid\infty}\avint_{[E^1]}\mathcal M_\Phi^{(v)}(g,\tau)\,\dd\tau
    +\sum_{\substack{v\nmid\infty\\v\ \mathrm{nonsplit}}}\log N_v\avint_{[E^1]}\mathcal M_\Phi^{(v)}(g,\tau)\,\dd\tau\\
    &+\sum_{0\ne x\in W^\perp(E)}i(x,\mathcal P)r(g)\Phi(x,0)\\
    &+\sum_{v\nmid\infty}\log N_v\sum_{0\ne x\in W^\perp(E)}m^0_{\Phi_v}(x)r(g^v)\Phi^v(x,0)\\
    &-2\sum_{v\in S_2}\log N_v\sum_{0\ne x\in W^\perp(E)}
    1_{\varpi_v\Lambda_v^\perp}(x)1_{\varpi_v^2\mathcal O_{F_v}^\times}(q(x))r(g^v)\Phi^v(x,0).
    \end{aligned}
\end{equation}
Moreover, $\mathcal H(g,\Phi)=\widehat{\mathcal Z}_*(g,\Phi)\cdot\mathcal P$ on these finite components. The sums over finite places in \eqref{height decomposition} are finite at each positive Fourier index.
\end{theorem}
\begin{proof}
Proposition~\ref{cuspidal height pairing} and \eqref{height P minus xi} identify $\mathcal H$ with the CM pairing. We decompose this pairing before evaluating the local lengths.

We apply the local decomposition of \cite[Sections~4.3 and~5.1]{Guo2}, using the uniformization of \cite[Section~6.1.2 and Proposition~6.1.5]{Qiu}. Homomorphisms with nonzero $W(v)$-component contribute $\mathcal M_\Phi^{(v)}$. The conjugate archimedean embeddings give the factor two, and the finite normalization gives $\log N_v$. Homomorphisms in $W_v^\perp$ with finite lifting locus contribute \eqref{perpendicular proper global sum}; those that lift over the CM component contribute \eqref{restricted divisor degree}, including $R_x$. Thus the intersection with the full lifting divisors is the sum of these contributions.

Changing the automorphic Green functions to normalized admissible functions contributes $-dE_{\BV,*}(0,g,\Phi)/n$ by \cite[Proposition~5.4]{Guo2}, hence zero. At $v\in S_2$, the first summand of \eqref{degenerate phi} imposes lifting of the divided homomorphism. Proposition~\ref{split arithmetic correction} gives its change from the full lifting divisor, namely the last line of \eqref{height decomposition}.
\end{proof}

The terms denoted by $\mathcal N_\phi^{(v)}$ in \cite[Proposition~8.3]{YZ} are included here in $m^0_{\Phi_v}$ for finite lifting loci and in $(R_x,\mathcal P)_{\mathrm{fin}}$ for homomorphisms that lift over the CM component. The latter are part of $i(x,\mathcal P)$.

\subsection{Local intersection numbers}\label{local intersection subsection}

This subsection computes the proper terms $\mathcal M_\Phi^{(v)}$ and $m^0_{\Phi_v}$ in Theorem~\ref{height decomposition theorem}. Proposition~\ref{known local intersections} recalls the archimedean calculation and the finite formulas for the embeddings used in \cite{Guo2,Qiu}. For a general $W$, Lemma~\ref{local product stabilization} identifies the actual formal modules, their polarizations, and the homomorphisms whose lifting lengths occur in the pairing. We then apply Lemma~\ref{canonical to quasi length} at inert and ramified places, and the ordinary deformation calculation at split places. The case $r_v=0$ at a ramified place uses the divisor $\mathcal Y(x)$ and is treated at the end.

\subsubsection*{The previously computed local terms}

The archimedean terms use the Green function $Q_{n,s}$ in \eqref{Q kernel}. At finite places, the CM datum in \cite[Sections~5.1.2--5.1.4]{Qiu} has $r_v=0$ at inert places and $r_v=1$ at ramified places. At a ramified place, the intersection of the one-dimensional subspace $E_ve_v^{(0)}$ with the $\varpi_{E_v}$-modular plane is $\varpi_{E_v}\mathcal O_{E_v}e_v^{(0)}$, where $e_v^{(0)}$ has unit norm; see \cite[Remark~6.1.8]{Qiu}. Both cases have $c_v=0$ in \eqref{local conductor definition}.

In the finite formulas below, we write $y=y_1+y_2$ with $y_1\in W_v^\perp$ and $0\ne y_2\in W(v)_v$, as in Section~\ref{standard finite subsection}.

\begin{proposition}\label{known local intersections}
\begin{enumerate}[(1)]
\item At a real place, the proper terms agree:
\begin{equation}\label{archimedean proper comparison}
    \overline{\mathcal K}_\Phi^{(v)}(g,\tau)=\mathcal M_\Phi^{(v)}(g,\tau).
\end{equation}
The value needed for the restricted divisor degree is
\begin{equation}\label{central limit}
    \lim_{\rho\to0^+}\left(Q_{n,0}(1+\rho)+\frac12\log\rho\right)=-\frac12H_{n-1},
\end{equation}
where $H_{n-1}=\sum_{j=1}^{n-1}j^{-1}$ and $H_0=0$.
\item At an inert place with $r_v=v(d_v)=0$, the standard CM datum gives
\begin{equation*}
    m_{\Phi_v}(y_1+y_2)=
       \begin{cases}
          (v(q(y_2))+1)1_{\Lambda_v^\perp}(y_1),&v(q(y_2))\ge0,\\
          0,&v(q(y_2))<0,
       \end{cases}
    \qquad m^0_{\Phi_v}=0.
\end{equation*}
At a split place with $r_v=0$ and $v\notin S_2$, the proper contribution on $W_v^\perp$ is zero. The ordinary decomposition described in Section~\ref{standard finite subsection} also excludes a proper term with nonzero $W_v$-component.
\item At a ramified place with the standard CM datum $r_v=1$, the proper multiplicity on $W_v^\perp$ is zero. For $y_2\ne0$,
\begin{equation*}
    m_{\Phi_v}(y_1+y_2)=
       \begin{cases}
           v(q(y_2))1_{\Lambda_v^\perp}(y_1),&v(q(y_2))\ge0,\\
           0,&v(q(y_2))<0.
       \end{cases}
\end{equation*}
\end{enumerate}
\end{proposition}
\begin{proof}
The archimedean assertions follow from \cite[Section~5.2]{Guo2}. More precisely, equation~(5.2.3) there gives
\begin{equation*}
    Q_{n,0}(1+\rho)=\frac12\left(\log\frac{1+\rho}{\rho}
       -\sum_{j=1}^{n-1}\frac{1}{j(1+\rho)^j}\right),
\end{equation*}
which proves \eqref{central limit}. Proposition~5.5 of that paper expresses the difference of the two sums as a multiple of $E_{\BV,*}(0,g,r(\tau)\Phi)$. The action by $\tau$ commutes with the Weil representation and fixes the origin. Proposition~\ref{two place proposition} therefore makes this Eisenstein series zero, giving \eqref{archimedean proper comparison}.

At finite places, the formal uniformization and product description in \cite[Section~6.1.2 and Proposition~6.1.5]{Qiu} reduce the standard CM datum to canonical formal modules. The canonical lifting formula \cite[Theorem~1.4]{Wew1}, with the normalization of \cite[Section~5.3]{Guo2}, gives the formulas of lengths. The split assertion is the ordinary calculation in that subsection. In the ramified case, the norm of the homomorphism in the Serre construction has the extra factor $\varpi_{E_v}^2$, giving $v(q(y_2))$ in place of $v(q(y_2))+1$.
\end{proof}

For a general $W$, the case $c_v>0$ requires a quasi-canonical lift. Moreover, the nonzero cosets in Proposition~\ref{lattice coset decomposition} may contribute to $m^0_{\Phi_v}$. The following discussion treats both changes.

\subsubsection*{Lifting homomorphisms on the CM deformation}

We now extend the computations above to general $W$. For this purpose, we first recall the lifting formula for quasi-canonical formal modules and then identify the formal module associated with $W$.

At a nonsplit place $v$, write $B/F_v$ for the division quaternion algebra and normalize valuation $v_B$ by assigning value one to a quaternionic uniformizer. We fix a canonical lift $G_0$ with multiplication by $\mathcal O_{E_v}$, which is a Lubin--Tate formal $\mathcal O_{E_v}$-module of height one. For $c\ge0$, let $G_c$ be a quasi-canonical lift with endomorphism order $\mathcal O_{F_v}+\varpi_v^c\mathcal O_{E_v}$, and choose a minimal cyclic isogeny $\psi_c:G_0\to G_c$ of relative degree $N_v^c$. Note that we work over a common finite extension $R$ of $\mathcal O_{\breve F_v}$, and write $e(R/F_v)=v_R(\varpi_v)$. The Tate lattice of $G_c$ is $\mathcal O_{E_v}+\varpi_v^{-c}\mathcal O_{F_v}$, whose maximal $\mathcal O_{E_v}$-stable sublattice is $\mathcal O_{E_v}$. Here we call the lattice in the rational Tate module by Tate lattice. Hence \cite[Proposition~1.1]{Wew2} identifies the reductions of the homomorphisms lifting over $R$ with
\begin{equation*}
    H_c=\overline\psi_c\mathcal O_{E_v}\subset\mathcal O_B,\qquad v_B(\overline\psi_c)=c.
\end{equation*}
For $h\in B\setminus H_c$, define the distance function
\begin{equation}\label{distance to Hom}
    \ell_c(h)=\max_{a\in H_c}v_B(h-a).
\end{equation}
For $h\in\mathcal O_B$, let $I_h\subset R$ be the ideal defining the locus over which $h$ lifts to a homomorphism, which we call the lifting ideal.

The normalized multiplicity is $2\length_R(R/I_h)/e(R/F_v)$. At an inert place, the factor two comes from the residue cardinality $N_v^2$ of $E_v$. At a ramified place, the two Serre embeddings give two CM branches of equal weight $1/2$; see \cite[(6.3.1)--(6.3.2)]{Chen} and \cite[Section~5.3]{Guo2}. Transporting the homomorphism between the two embeddings preserves $H_c$ and $\ell_c$, so the two lengths agree. Their average is
\begin{equation*}
    \frac12\sum_{\pm}\frac{\length_R(R/I_h)}{e(R/E_v)}\log N_v=\frac{2\length_R(R/I_h)}{e(R/F_v)}\log N_v.
\end{equation*}
The equality uses $e(E_v/F_v)=2$. This normalization is unchanged by a further finite extension.

\begin{lemma}\label{canonical to quasi length}
For $h\in\mathcal O_B\setminus H_c$, with $\ell=\ell_c(h)$, one has
\begin{equation}\label{uniform Hom length}
    \frac{2\length_R(R/I_h)}{e(R/F_v)}=\begin{cases}\displaystyle\frac{2N_v^{1-c}(N_v^{\ell+1}-1)}{N_v^2-1},&0\le\ell<c,\quad v\text{ inert},\\[2mm]\displaystyle\frac{N_v^{-c}(N_v^{\ell+1}-1)}{N_v-1},&0\le\ell<c,\quad v\text{ ramified},\\[2mm]\ell+1-a_v(c),&\ell\ge c.\end{cases}
\end{equation}
A quasi-homomorphism outside $\mathcal O_B$ has empty lifting locus, and the elements of $H_c$ lift over $R$.
\end{lemma}
\begin{proof}
This is the lifting formula of \cite[Theorem~2.4]{Mih} with $i=0$ and $j=c$. Proposition~2.1 there identifies its Hom lattice with $H_c$, hence its distance with \eqref{distance to Hom}; the lifting colength is denoted there by $\alpha+1$.

In the formula of Mihatsch, the ramification factor is
\begin{equation*}
    e_c=\begin{cases}
        1,&v\text{ inert},\ c=0,\\
        (N_v+1)N_v^{c-1},&v\text{ inert},\ c>0,\\
        2N_v^c,&v\text{ ramified}.
    \end{cases}
\end{equation*}
The specialization of the cited formula, divided by $e(R/F_v)$, is
\begin{equation*}
    \frac{\length_R(R/I_h)}{e(R/F_v)}=
    \begin{cases}
        \displaystyle e_c^{-1}\sum_{j=0}^{\ell}N_v^j,&\ell<c,\\[2mm]
        \displaystyle e_c^{-1}\sum_{j=0}^{c-1}N_v^j+\frac{\ell+1-c}{2},&\ell\ge c.
    \end{cases}
\end{equation*}
An empty sum is zero. This formula includes $\ell=0$, for which \cite[Lemma~2.3]{Mih} gives the intersection of the two quasi-canonical lifts. At $c=0$, it is the canonical lifting formula of \cite[Theorem~1.4]{Wew1}. We also refer to \cite[Proposition~7.2]{KR} for the inert case.

Multiplication by two gives the normalization in \eqref{arithmetic normalization}. The sum above proves the formula for $\ell<c$; for $\ell\ge c$, the required identity is
\begin{equation*}
    2e_c^{-1}\sum_{j=0}^{c-1}N_v^j=c-a_v(c),
\end{equation*}
which also holds at $c=0$. Note that the equivalence in \cite[Theorem~4.9]{LMZ} preserves the lifting locus under base change; its duality and Lie quotients are identified in Theorem~4.20 and Proposition~7.8 of that paper. Finally, a nonintegral quasi-homomorphism cannot lift on the special fiber, and elements of $H_c$ lift by its definition.
\end{proof}

To apply Lemma~\ref{canonical to quasi length} to $m_{\Phi_v}$ and $m^0_{\Phi_v}$, we must identify their defining homomorphisms on the formal module $G_{c_v}$. We now identify the formal module associated with the CM cycle using the rank-two summand $\Lambda_{1,v}$ in Lemma~\ref{all place line types}. At a ramified place, we assume $r_v\ge1$. In addition to the quadratic order of Proposition~\ref{local CM order}, the comparison keeps the actual Tate lattice and polarization. In the following, a framing means a quasi-isogeny from the special fiber of the formal module to the fixed formal module with endomorphism algebra $B$. It allows the homomorphisms to be written as elements of $B$.

\begin{lemma}\label{local product stabilization}
After a finite extension, the relative $p$-divisible group over a normalized CM component decomposes as the polarized product of the factor of rank two associated with $\Lambda_{1,v}$ and the positive canonical factors associated with $\Lambda_{1,v}^\perp$. At a nonsplit place, the first factor is the Serre tensor of $G_{c_v}$.

Under this identification and the given polarization, the special homomorphisms in the factor of rank two are represented by $h\in B$, with
\begin{equation}\label{quaternion norm shift}
    v(q(h))=v_B(h)\quad(v\text{ inert}),\qquad v(q(h))=1+v_B(h)\quad(v\text{ ramified}).
\end{equation}
Moreover, the homomorphisms lifting over the whole CM component form $\Lambda_v^\perp$. For a fixed homomorphism whose positive components are integral, the lifting ideal is the ideal of its component in the factor of rank two. If a positive component is nonintegral, the lifting locus is empty.
\end{lemma}
\begin{proof}
The proof is divided into several steps. In the first step we prove the polarized product decomposition. The second identifies its factor of rank two with the Serre construction for $G_{c_v}$ and identifies the homomorphisms that lift over the base. The last step checks \eqref{quaternion norm shift} for the given polarization and deduces the assertions about lifting ideals.

\smallskip
\noindent\textbf{Step 1.} The decomposition $\Lambda_v=\Lambda_{1,v}\perp\Lambda_{1,v}^\perp$ gives integral orthogonal projections. Since $\U(W_v)$ fixes $W_v^\perp$ pointwise, these projections commute with its action. By the CM description of the point, the Galois action on the Tate module preserves this decomposition after a finite extension. Thus the two projections extend to the $p$-divisible group by Tate's theorem \cite[Theorem~4]{Tate}. Their idempotent and self-adjoint identities extend by uniqueness, giving the polarized product. The definite factors have canonical CM lifts, as in \cite[Section~5.3]{Guo2} and \cite[Sections~6.1.1--6.1.2]{Qiu}; every integral homomorphism into them lifts over the base.

\smallskip
\noindent\textbf{Step 2.} We identify the formal module of relative height two by the polarized Serre construction. At a ramified place, the polarized construction is given in \cite[Proposition~6.3, Lemma~6.5, and Remark~6.7]{RSZ1}. The correspondence for a fixed homomorphism is \cite[Lemma~4.6.2 and Remark~4.6.7]{Chen}.

Recall that the matrices in Proposition~\ref{local CM order} determine the Tate lattice. At an inert place, multiplication by the trace-zero unit $\delta$ sends the lattice generated by $1$ and $2/(q(w_v)\delta)$ to
\begin{equation*}
    \mathcal O_{F_v}\delta+\frac{2}{q(w_v)}\mathcal O_{F_v}=\mathcal O_{E_v}+\varpi_v^{-r_v}\mathcal O_{F_v}.
\end{equation*}
In the ramified case, multiplication of the lattice generated by $1$ and $-2\varpi_{E_v}/q(w_v)$ by $\varpi_{E_v}$ gives
\begin{equation*}
    \mathcal O_{F_v}\varpi_{E_v}-\frac{2\varpi_{E_v}^2}{q(w_v)}\mathcal O_{F_v}=\mathcal O_{E_v}+\varpi_v^{-(r_v-1)}\mathcal O_{F_v}.
\end{equation*}
Thus its maximal $\mathcal O_{E_v}$-stable sublattice is $\mathcal O_{E_v}$, and the quotient is isomorphic to $\mathcal O_{F_v}/\varpi_v^{c_v}\mathcal O_{F_v}$. By \cite[Theorem~3.2 and Lemma~3.4]{Wew1}, the resulting formal module is $G_{c_v}$. The correspondence in \cite[Remark~4.6.7]{Chen} identifies the lifting ideals of the fixed homomorphism. In particular, the homomorphisms that lift over the base form $H_{c_v}$. Note that at an inert place, the reduction of $\psi_{c_v}$ commutes with the $\mathcal O_{E_v}$-action when $c_v$ is even, and with its Galois conjugate when $c_v$ is odd, by \cite[Corollary~2.2]{Mih}.

\smallskip
\noindent\textbf{Step 3.} We compare the Hermitian pairings in these coordinates. At a ramified place the correspondence for the Serre tensor sends $h$ to the column
\begin{equation*}
    \binom{h\iota_0(\varpi_{E_v})}{h},\qquad q\left(\binom{h\iota_0(\varpi_{E_v})}{h}\right)=-2\varpi_{E_v}^2\operatorname{Nrd}(h);
\end{equation*}
see \cite[(3.5)--(3.6) and Lemma~3.5]{RSZ1}. Its second coordinate shows that the column lifts exactly when $h$ lifts. This proves the ramified equality in \eqref{quaternion norm shift}.

To give the generator of $H_{c_v}$ the norm $-q(w_v)$, choose $\beta\in\mathcal O_B^\times$ with
\begin{equation*}
    \operatorname{Nrd}(\beta)=\frac{q(w_v)}{2\varpi_{E_v}^2\operatorname{Nrd}(\overline\psi_{c_v})}.
\end{equation*}
The right side is a unit, and $\operatorname{Nrd}(\mathcal O_B^\times)=\mathcal O_{F_v}^\times$. The change of framing by $\beta$ multiplies the polarization by $\operatorname{Nrd}(\beta)$, as in \cite[Remark~3.6 and proof of Proposition~6.3]{RSZ1}. In these coordinates the same homomorphism is $\beta h$, its norm is $-2\varpi_{E_v}^2\operatorname{Nrd}(\beta h)$, and the integral Hom lattice is $\beta H_{c_v}$. Thus $\beta\overline\psi_{c_v}$ has norm $-q(w_v)$, the lifting ideal is unchanged, and
\begin{equation*}
    \max_{a\in\beta H_{c_v}}v_B(\beta h-a)
    =\max_{a\in H_{c_v}}v_B(h-a).
\end{equation*}
Since the polarization changes only by a unit in $\mathcal O_{F_v}^{\times}$, we can rescale one of the isotropic basis vectors to recover the original Hermitian pairing. Thus the polarized lattice agrees with $\Lambda_{1,v}$, with $\beta\overline\psi_{c_v}$ corresponding to $t_v$. At an inert place, the same normalization uses the surjective norm on $\mathcal O_{E_v}^\times$.

Finally, a homomorphism to the product lifts if and only if each component lifts. Every integral homomorphism to the positive factors lifts, while the homomorphisms to the factor of rank two that lift over the CM component form $H_{c_v}$. Together they give $\Lambda_v^\perp$. Moreover, when the positive components are integral, the lifting ideal is determined by the factor of rank two. This proves the last assertions.
\end{proof}

Lemma~\ref{local product stabilization} reduces the calculation of $m_{\Phi_v}$ and $m^0_{\Phi_v}$ to the factor of rank two: the positive components contribute $1_{\Lambda_{1,v}^\perp}$. At inert and ramified places, the distance $\ell_{c_v}$ to $H_{c_v}$ and \eqref{uniform Hom length} now give the required lengths. At split places we use the ordinary lifting equation. These lengths will be compared with Propositions~\ref{zero Whittaker coefficients} and~\ref{nonzero Whittaker coefficients}.

\subsubsection*{Inert places}

We now compute $m_{\Phi_v}$ and $m^0_{\Phi_v}$ at an inert place, using the result in Lemma~\ref{local product stabilization}. Assume $v(d_v)=0$ first, so the analytic coefficients have the normalization of Proposition~\ref{nonzero Whittaker coefficients}. In the factor of rank two, a nonzero class in \eqref{explicit plane cosets} of denominator $\varpi_v^k$ has norm valuation $r_v-2k$. Its distance from $H_{c_v}=H_{r_v}$ is therefore $r_v-2k$. The two orthogonal norm classes have opposite valuation parity, so for $y_2\ne0$,
\begin{equation}\label{inert distances}
    \ell_{c_v}(y_1+y_2)=\begin{cases}v(q(y_2)),&y_1\in H_{c_v},\\\min\{r_v-2k,v(q(y_2))\},&y_1\text{ has denominator }\varpi_v^k.\end{cases}
\end{equation}
A negative value means that the homomorphism is nonintegral and its lifting locus is empty.

\begin{proposition}\label{inert general matching}
For every $r_v\ge0$ and $y_2\ne0$,
\begin{equation*}
    2k_{\Phi_v}(1,y_1+y_2)=m_{\Phi_v}(y_1+y_2)\log N_v.
\end{equation*}
The proper multiplicity on $W_v^\perp$ is zero on $\Lambda_v^\perp$. On a nonzero class of denominator $\varpi_v^k$, it is
\begin{equation*}
    m^0_{\Phi_v}(y_1)=\begin{cases}\displaystyle\frac{2N_v^{1-r_v}(N_v^{r_v-2k+1}-1)}{N_v^2-1},&2k\le r_v,\\[2mm]0,&2k>r_v.\end{cases}
\end{equation*}
Consequently,
\begin{equation*}
    \bar k_{\Phi_v}(1,y_1)+m^0_{\Phi_v}(y_1)\log N_v-c_{\Phi_v}(1,y_1)=a_v(r_v)\Phi_v(y_1,0)\log N_v.
\end{equation*}
\end{proposition}
\begin{proof}
We substitute the distances in \eqref{inert distances} into the lifting formula \eqref{uniform Hom length}. On the zero coset the result is the coefficient in \eqref{inert opposite derivative}; on a nonzero coset it is the coefficient in Proposition~\ref{nonzero Whittaker coefficients}. This proves the equality $2k_{\Phi_v}=m_{\Phi_v}\log N_v$, including zero when the distance is negative, and gives $\bar k_{\Phi_v}=0$.

To compute the terms on $W_v^\perp$, we set $y_2=0$ and use the explicit expression of $c_{\Phi_v}$ in \eqref{c coset expansion}. A nonzero coset has finite lifting locus and distance $r_v-2k$. Formula~\eqref{uniform Hom length} gives the formula for $m^0_{\Phi_v}$, which equals \eqref{inert coset zero derivative} divided by $\log N_v$. On the zero coset, $m^0_{\Phi_v}=0$ and $c_{\Phi_v}=-a_v(r_v)\log N_v$. These identities prove the last identity in the statement.
\end{proof}

When $v(d_v)>0$, we use the analytic coefficients at the end of Section~\ref{local Whittaker subsection}. Their exponents are shifted by $v(d_v)$, while the arithmetic conductor remains $r_v$.

\begin{corollary}\label{inert different boundary}
The difference $\bar k_{\Phi_v}$ in \eqref{general proper difference} has a Schwartz extension, and
\begin{equation*}
    \bar k_{\Phi_v}(1,y_1)+m^0_{\Phi_v}(y_1)\log N_v-c_{\Phi_v}(1,y_1)=\bigl(v(d_v)+a_v(r_v)\bigr)\Phi_v(y_1,0)\log N_v.
\end{equation*}
\end{corollary}
\begin{proof}
On the zero coset, for sufficiently large $v(q(y_2))$, the difference between the analytic and arithmetic coefficients is
\begin{equation*}
    \bigl(v(d_v)+a_v(r_v)-a_v(r_v+v(d_v))\bigr)\log N_v.
\end{equation*}
Note that we have $m^0_{\Phi_v}=0$ and $c_{\Phi_v}(1,y_1)=-a_v(r_v+v(d_v))\log N_v$ on the zero coset. This proves the desired identity. On every nonzero class, the distance function \eqref{inert distances} makes the proper length constant near $y_2=0$, with value $m^0_{\Phi_v}(y_1)$. Since the norm is bounded away from zero on the corresponding coset in $W_v$, Propositions~\ref{nonzero Whittaker coefficients} and~\ref{Whittaker chosen data} make the nonzero derivative constant with value $c_{\Phi_v}(1,y_1)$. Then their difference gives the identity. Finally, the finitely many formulas have compact support and are locally constant, including at $y_2=0$, so the extension is Schwartz.
\end{proof}

\subsubsection*{Split places outside $S_2$}

At a split place, the ordinary decomposition leaves only the proper term $m^0_{\Phi_v}$ on $W_v^\perp$, as explained in Section~\ref{standard finite subsection}. This term can be computed in the basis $w_v,t_v$ of Section~\ref{local line conductor subsection}. In the factor of rank two, we write $y_1=t_v(\xi_+,\xi_-)$. The zero coset has both coordinates integral, and the quotient in \eqref{all place coset decomposition} is $(\varpi_v^{-r_v}\mathcal O_{F_v}/\mathcal O_{F_v})^2$.

At one place of $E$ above $v$, the ordinary factor is an extension of an \'etale module by a formal module of height one. By Proposition~\ref{local CM order} and the Serre--Tate criterion, its extension parameter $t_{r_v}$ has exact $\varpi_v^{r_v}$-order; see \cite[Section~6.2, equations~(6.2.2)--(6.2.5)]{Chen}. The homomorphism is integral on the special fiber precisely when
\begin{equation*}
    q(w_v)\xi_+\in\mathcal O_{F_v},\qquad\xi_-\in\mathcal O_{F_v},
\end{equation*}
and its lifting equation is $[q(w_v)\xi_+](t_{r_v})=0$. The two coordinates are interchanged at the other place. These conditions follow by writing the homomorphism in the original self-dual lattice, as in the ordinary calculation of \cite[Section~5.3]{Guo2}.

\begin{proposition}\label{split general matching}
There is no proper term with nonzero $W_v$-component. A class in \eqref{all place coset decomposition} with one coordinate of exact denominator $\varpi_v^k$, $1\le k\le r_v$, and the other coordinate integral has
\begin{equation*}
    m^0_{\Phi_v}(y_1)=\frac{1}{N_v^{k-1}(N_v-1)}.
\end{equation*}
The multiplicity is zero on the zero coset and on classes with both coordinates nonintegral. Moreover,
\begin{equation*}
    m^0_{\Phi_v}(y_1)\log N_v-c_{\Phi_v}(1,y_1)=\bigl(r_v+v(d_v)\bigr)\Phi_v(y_1,0)\log N_v.
\end{equation*}
\end{proposition}
\begin{proof}
As in \cite[Section~5.3]{Guo2}, the decomposition into \'etale and connected parts forces the $W_v$-component to vanish. If $\xi_+$ has denominator $\varpi_v^k$ and $\xi_-$ is integral, then $[q(w_v)\xi_+](t_{r_v})$ is a primitive torsion point of level $k$. Its normalized valuation is $1/(N_v^{k-1}(N_v-1))$, by the Newton polygon of $[\varpi_v](T)$; this is also the multiplicity used in \cite[proof of Lemma~8.7]{YZ}. At the other place of $E$ above $v$, the same calculation applies with the two coordinates interchanged. On the other hand, if both coordinates are nonintegral, the homomorphism fails the integrality condition at both places, so its lifting locus is empty. Finally, on the zero coset, the homomorphism lifts over the whole CM component.

The comparison with Proposition~\ref{zero Whittaker coefficients} is now explicit. If exactly one coordinate is nonintegral, its zero Whittaker derivative equals the proper length multiplied by $\log N_v$. If both are nonintegral, both terms vanish. If both are integral, the proper length is zero and the zero Whittaker derivative is $-(r_v+v(d_v))\log N_v$. This proves the identity in the statement.
\end{proof}

\subsubsection*{Ramified places with $r_v\ge1$}

At a ramified place with $r_v\ge1$, we apply \eqref{uniform Hom length} with $c_v=r_v-1$. For $y_2\ne0$, the corresponding values of the distance function $\ell_{c_v}$ are
\begin{equation*}
    \ell_{c_v}(y_1+y_2)=
    \begin{cases}
        v(q(y_2))-1,&y_1\in H_{c_v},\\
        \min\{r_v-k,v(q(y_2))\}-1,
        &y_1\text{ lies in a nonzero coset of order }\varpi_{E_v}^k.
    \end{cases}
\end{equation*}
These formulas follow from \eqref{explicit plane cosets} and \eqref{quaternion norm shift}, since the two terms of equal valuation cannot cancel by the local norm condition.

\begin{proposition}\label{ramified general matching}
For $y_2\ne0$, the proper multiplicity on the zero coset is
\begin{equation*}
    m_{\Phi_v}(y_1+y_2)=\begin{cases}\displaystyle N_v^{1-r_v}\frac{N_v^{v(q(y_2))}-1}{N_v-1},&0\le v(q(y_2))<r_v,\\[2mm]v(q(y_2))-a_v(r_v-1),&v(q(y_2))\ge r_v,\\0,&v(q(y_2))<0.\end{cases}
\end{equation*}
On a nonzero class of denominator $\varpi_{E_v}^k$, it is
\begin{equation*}
    m_{\Phi_v}(y_1+y_2)=\begin{cases}\displaystyle N_v^{1-r_v}\frac{N_v^{\min\{r_v-k,v(q(y_2))\}}-1}{N_v-1},&k\le r_v,\ v(q(y_2))\ge0,\\[2mm]0,&\text{otherwise}.\end{cases}
\end{equation*}
The proper multiplicity on $W_v^\perp$ vanishes on the zero coset and is
\begin{equation}\label{ramified perp proper table}
    m^0_{\Phi_v}(y_1)=\begin{cases}\displaystyle N_v^{1-r_v}\frac{N_v^{r_v-k}-1}{N_v-1},&1\le k<r_v,\\[2mm]0,&k\ge r_v\end{cases}
\end{equation}
on a nonzero class. The Schwartz extension of \eqref{general proper difference} is
\begin{equation}\label{ramified full Schwartz difference}
    \bar k_{\Phi_v}(1,y_1+y_2)=N_v^{-r_v}1_{\varpi_{E_v}^{-r_v}\mathcal O_{E_v}t_v\oplus\Lambda_{1,v}^\perp}(y_1)1_{\mathcal O_{F_v}}(q(y_2))\log N_v.
\end{equation}
Consequently,
\begin{equation*}
    \bar k_{\Phi_v}(1,y_1)+m^0_{\Phi_v}(y_1)\log N_v-c_{\Phi_v}(1,y_1)=\bigl(1+a_v(r_v-1)\bigr)\Phi_v(y_1,0)\log N_v.
\end{equation*}
\end{proposition}
\begin{proof}
We apply Lemma~\ref{canonical to quasi length} to the explicit distance function above, with $c_v=r_v-1$. This gives the two formulas for $m_{\Phi_v}$, including zero when the distance is negative. For $y_2=0$, the nonzero classes remain proper and have distance $r_v-k-1$, proving \eqref{ramified perp proper table}. In particular, the classes $k=r_v$ give no intersection.

Next, we subtract these lengths from the Whittaker derivatives in Proposition~\ref{nonzero Whittaker coefficients} and \eqref{ramified centered nonzero}. On the zero coset the difference divided by $\log N_v$ is $N_v^{-r_v}$ whenever $v(q(y_2))\ge0$. For a nonzero class, the same difference is $N_v^{-r_v}$ precisely when $k\le r_v$ and $v(q(y_2))\ge0$, and is zero otherwise. Multiplication by $1_{\Lambda_{1,v}^\perp}$ gives \eqref{ramified full Schwartz difference}, whose support is a lattice.

Finally, on a nonzero class with $k\le r_v$, \eqref{ramified coset zero} and \eqref{ramified perp proper table} give
\begin{equation*}
    N_v^{-r_v}+N_v^{1-r_v}\frac{N_v^{r_v-k}-1}{N_v-1}-N_v^{-r_v}\frac{N_v^{r_v-k+1}-1}{N_v-1}=0.
\end{equation*}
All three terms vanish when $k>r_v$. On the zero coset their sum is $N_v^{-r_v}+a_v(r_v)=1+a_v(r_v-1)$, proving the last identity in the statement for every $r_v\ge1$.
\end{proof}

\subsubsection*{Ramified places with $r_v=0$}

It remains to compute the proper terms at $v\in S_0(W)$, as defined in \eqref{S0 definition}. Here $n\ge2$ is even, $M_v$ is self-dual and $\Phi_v=1_{\Lambda_v^\vee}$. Let $G$ be the canonical formal $\mathcal O_{F_v}$-module with positive $\mathcal O_{E_v}$-action, and let $\overline G$ be the same formal module with the conjugate action. The factor corresponding to $M_v$ is the canonical lift $\overline G$ with the conjugate $\mathcal O_{E_v}$-action. We first identify the horizontal divisor selected by this $\Phi_v$ and then compute its lifting length.

We use the polarization $\lambda:\mathcal A\to\mathcal A^\vee$ introduced in Section~\ref{integral model subsection}. The action on the dual object is the one for which $\lambda$ is $\mathcal O_{E_v}$-linear. For a fixed rational special homomorphism $x$, let $\mathcal Y(x)$ be the lifting locus of $\lambda x$, as in \cite[Definition~2.9]{Yao}.

\begin{lemma}\label{unit horizontal divisor}
On the formal completions along the CM cycle at $v$, the horizontal Cartier divisor selected by $1_{\Lambda_v^\vee}$ is
\begin{equation*}
    D_x=\mathcal Y(x).
\end{equation*}
The identification preserves the rational homomorphism $x$, the finite maps, and the stabilizer weights in the divisor series.
\end{lemma}
\begin{proof}
The proof is routine. We first identify the generic lifting condition. Then, we use flatness to identify its horizontal closure, and compare the finite maps.

The auxiliary principal polarization and Tate duality identify $\Lambda_v^\vee$ with the rational special homomorphisms for which $\lambda x$ is integral on Tate modules. On the relative factor at $v$ in characteristic zero, multiplying $\lambda x$ by a sufficiently large power of $\varpi_v$ gives an integral homomorphism. By the Tate module condition, this multiple vanishes on the torsion killed by that power of $\varpi_v$. Since this torsion is finite \'etale, the multiple factors through multiplication by the same power of $\varpi_v$. Thus the generic moduli conditions agree for the fixed rational homomorphism $x$.

We now apply the formal uniformization of \cite[Section~6.1.2 and Proposition~6.1.5]{Qiu}. Note that our signature convention agrees with that of \cite[Section~2.1]{Yao} after conjugating the $\mathcal O_{E_v}$-actions on both $\mathcal A_0$ and $\mathcal A^\vee$. This does not affect whether $\lambda x$ lifts. In odd Hermitian rank $n+1\ge3$, $\mathcal Y(x)$ is a relative Cartier divisor by \cite[Remark~1.10 and Proposition~5.17]{Yao}. The same assertion holds for the other framing class after changing the polarization by a unit, as in \cite[Section~2.3]{Yao}. Since $\mathcal Y(x)$ is flat over $\mathcal O_{E_v}$, its defining ideal on each completed local ring is the contraction of its generic-fiber ideal. Thus the generic identification above proves $D_x=\mathcal Y(x)$ as an equality of Cartier divisors.

Finally, this identification preserve the multiplicities in the height series. Indeed, its inverse on the generic moduli scheme is composition with $\lambda^{-1}$, and both operations commute with the level action and the forgetful maps. The stabilizers are therefore unchanged. Passing to flat closures and applying the projection formula gives the multiplicities in \eqref{arithmetic normalization}.
\end{proof}

We have the following result.

\begin{proposition}\label{unit local matching}
For every even $n\ge2$ and $y=y_1+y_2\in W_v^\perp\perp W(v)_v$ with $y_2\ne0$,
\begin{equation*}
    m_{\Phi_v}(y)=
    \begin{cases}
        \bigl(v(q(y_2))+1\bigr)\Phi_v^\perp(y_1),&v(q(y_2))\ge0,\\
        0,&v(q(y_2))<0.
    \end{cases}
\end{equation*}
Moreover,
\begin{equation*}
    2k_{\Phi_v}(1,y)=m_{\Phi_v}(y)\log N_v,\qquad
    m^0_{\Phi_v}=0,\qquad c_{\Phi_v}(1,\cdot)=0.
\end{equation*}
Thus $\bar k_{\Phi_v}$ extends by zero to all of $V(v)_v$.
\end{proposition}
\begin{proof}
We first identify the factors of the polarized $p$-divisible group over the CM component. Then we use the canonical lifting formula to compute the local intersection number and compare it with the Whittaker derivative.

Note that the orthogonal projections onto $M_v$ and each $\varpi_{E_v}$-modular summand of $\Lambda_v^\perp$ are $\mathcal O_{E_v}$-linear and preserve $\Lambda_v$. After a finite extension, Tate's extension theorem applies as in Lemma~\ref{local product stabilization}. Thus the $p$-divisible group over the CM component decomposes into the canonical lift $\overline G$ corresponding to $M_v$ and the canonical lifts corresponding to the summands of $\Lambda_v^\perp$. Their polarizations are induced by the original polarization, and the polarization on $\overline G$ is principal up to a unit. For the positive factors, the Tate duality in Lemma~\ref{unit horizontal divisor} identifies the lattice of maps $\lambda y_1$ with $(\Lambda_v^\perp)^\vee$.

Now, it is clear that for the factors corresponding to $\Lambda_v^\perp$, the homomorphism $\lambda y_1$ is integral if and only if $y_1\in(\Lambda_v^\perp)^\vee$. Since every such integral homomorphism lifts over the whole CM component, we obtain the factor $\Phi_v^\perp(y_1)$ and $m^0_{\Phi_v}=0$.

Next, we identify the underlying formal modules of $G$ and $\overline G$. In their division quaternion algebra choose $j\in\mathcal O_B^\times$ with $ja=\bar a j$ for $a\in E_v$ and $j^2\in\mathcal O_{F_v}^\times$ not a norm. Then
\begin{equation*}
    \mathcal O_B=\mathcal O_{E_v}\oplus\mathcal O_{E_v}j,
    \qquad W(v)_v=E_vj.
\end{equation*}
The Hermitian norm of the negative component differs from the reduced norm by a unit, since its polarization is principal. For the homomorphism $h\in E_vj$ representing $\lambda y_2$,
\begin{equation*}
    v_B(h)=v(q(y_2)),\qquad
    \max_{a\in\mathcal O_{E_v}}v_B(h-a)=v_B(h).
\end{equation*}
Indeed, the residues of $1$ and $j$ are linearly independent over the residue field of $F_v$. Hence $v_B(h-a)=\min\{v_B(h),v_B(a)\}$ for every $a\in\mathcal O_{E_v}$, proving the second equality.

Now we prove the formula for $m_{\Phi_v}$. In the case $v(q(y_2))<0$ the lifting locus is empty. Otherwise, the canonical formula in \cite[Theorem~1.4]{Wew1} or equivalently \eqref{uniform Hom length} with $c=0$, gives normalized length $v(q(y_2))+1$. Since $e(R/F_v)=2e(R/E_v)$, this is exactly the coefficient of $\log N_v$ in \eqref{arithmetic normalization}. Lemma~\ref{unit horizontal divisor} then identifies the lifting divisor with the horizontal closure.

Finally, note that equations~\eqref{ramified centered family} and~\eqref{ramified centered nonzero} at $r_v=0$ give $c_{\Phi_v}=0$ and
\begin{equation*}
    2k_{\Phi_v}(1,y)=\bigl(v(q(y_2))+1\bigr)\Phi_v^\perp(y_1)\log N_v
    \qquad\bigl(v(q(y_2))\ge0\bigr),
\end{equation*}
with zero for negative valuation. This proves the identities in the statement. 
\end{proof}

This completes the proper terms at $S_0(W)$. The remaining contribution in \eqref{height decomposition} is the restricted degree $i(x,\mathcal P)$, which we compute in the next subsection.

\subsection{Divisors through the CM cycle}\label{adjunction subsection}

The remaining term in Theorem~\ref{height decomposition theorem} is $i(x,\mathcal P)$, defined in \eqref{restricted divisor degree}. We compare $\mathcal O(D_x)$ with $\mathcal L^{-1}$, first outside $S_0(W)$ and then at $S_0(W)$, and use the regularized metric in \eqref{central limit} to compute its degree. Finally, at each $v\in S_2$, we compute the correction arising
from requiring $x/\varpi_v$ to lift instead of $x$.

\subsubsection*{The integral comparison outside $S_0(W)$}

At these places, Lemma~\ref{local product stabilization} identifies the Hodge module on the CM component with the one in the factor of rank two. We use this to apply the conormal calculation of \cite[Lemma~5.1]{Guo2} to the present embedding.

\begin{lemma}\label{Hodge stabilization}
At a finite place $v\notin S_0(W)$, for a fixed homomorphism $x$ which lifts over the whole CM component, the polarization comparison gives an integral isomorphism
\begin{equation*}
    \mathcal O(D_x)|_{\mathcal P}\simeq\mathcal L^{-1}|_{\mathcal P}.
\end{equation*}
Moreover, under the product decomposition of Lemma~\ref{local product stabilization}, the inverse Hodge line bundle comes from the factor of rank two.
\end{lemma}

\begin{proof}
The main idea is to check the image module defining $\mathcal L^{-1}$ and apply the lifting obstruction. At an inert place, the module $(J_{\mathfrak A_v}\Lie\mathcal A)_1$ comes entirely from the factor of rank two in Lemma~\ref{local product stabilization}, since the remaining factors contribute zero. At a split place, the same conclusion follows by considering the two idempotent components of $E_v=F_v\times F_v$. Also recall that when $F_v/\Q_p$ is ramified, we use \cite[Proposition~7.8]{LMZ} to identify the corresponding Lie quotient and polarization map on the relative model.

It remains to consider the ramified places with $r_v\ge1$. In this case, the operator $\iota_{\mathcal A}(\varpi_{E_v})-\varpi_{E_v}$ on the factor of rank two has matrix
\begin{equation*}
    \begin{pmatrix}
        -\varpi_{E_v}&\varpi_{E_v}^2\\
        1&-\varpi_{E_v}
    \end{pmatrix}
    =\binom{-\varpi_{E_v}}{1}(1,-\varpi_{E_v})
\end{equation*}
with respect to the basis $1,\varpi_{E_v}$. Its image is generated by $(-\varpi_{E_v},1)^t$, independently of the conductor, and it vanishes on the positive factors. Since the second coordinate of $(-\varpi_{E_v},1)^t$ is a unit, the image is a direct summand of rank one on this local chart. In particular, it equals its double dual. Then for the $\varpi_{E_v}$-modular model, \cite[Lemma~2.41]{LL} identifies it with the module in the lifting obstruction. While for the almost $\varpi_{E_v}$-modular model, the conormal argument of \cite[Lemma~5.1]{Guo2} applies on this open subset.

Now, for the $\varpi_{E_v}$-modular case, the lifting obstruction is described in \cite[Lemmas~2.37, 2.39--2.41]{LL}. In the almost $\varpi_{E_v}$-modular case with $r_v\ge1$, the preceding calculation shows that $\mathcal I_{\mathcal A}^{\vee\vee}=\mathcal I_{\mathcal A}$ on this local chart. Thus the conormal comparison of \cite[Lemma~5.1]{Guo2} also applies. In both cases, the lifting obstruction identifies the conormal line bundle $\mathcal O(-D_x)|_{D_x}$ with $\mathcal L|_{D_x}$. Dualizing and restricting to $\mathcal P$ proves the asserted isomorphism, with the same identification on the generic fiber.
\end{proof}

\subsubsection*{The integral comparison at $S_0(W)$}

For $v\in S_0(W)$, the divisor $D_x$ is defined by the lifting of $\lambda x:\mathcal A_0\to\mathcal A^\vee$. Here $D_x=\mathcal Y(x)$ by Lemma~\ref{unit horizontal divisor}, so its lifting obstruction belongs to the Lie algebra of $\mathcal A^\vee$. However, the Hodge bundle is defined from the image module on $\mathcal A$. We compare these two modules before computing the degree of the restricted line bundle.

For the next lemma and proposition, to be consistent with the conventions in the literature we write $u=\varpi_{E_v}$, and let $\lambda':\mathcal A^\vee\to\mathcal A$ satisfy $\lambda'\lambda=\iota_{\mathcal A}(u)$. Its differential will give the integral comparison. A frame below means a choice of bases for the locally free modules, and a polarized frame is chosen with the specified pairing matrices.

\begin{lemma}\label{unit image calculation}
On a completed local chart through the CM component, with coordinate ring $R$, the following modules are free of rank one. Their generators can be chosen so that
\begin{equation}\label{unit image modules}
    \mathcal I_{\mathcal A}^{\vee\vee}=R\cdot e,\qquad
    \ker\bigl(\iota_{\mathcal A^\vee}(u)+u\mid\Lie\mathcal A^\vee\bigr)=R\cdot e^\vee,
    \qquad d\lambda'(e^\vee)=-e.
\end{equation}
Moreover, $R\cdot e^\vee$ is a direct summand of $\Lie\mathcal A^\vee$. On the CM component, the image of $R\cdot e$ in the negative rational Lie summand is $u\Lie\overline G$.
\end{lemma}
\begin{proof}
These formulas actually follow from an explicit local model with the given polarization. We first identify its coordinates, then compute the two modules and the differential of $\lambda'$. Let $J$ be the block diagonal alternating matrix of size $n$, with $J^t=-J$ and $J^2=-I_n$. Set
\begin{equation*}
    R=\mathcal O_{\breve E_v}[[b_1,\ldots,b_n]],\qquad
    b=(b_1,\ldots,b_n),\qquad
    e=\binom{\frac12Jb^t}{u},\qquad
    e^\vee=\binom{\frac12Jb^t}{1}.
\end{equation*}
In a polarized frame, the two Lie actions are
\begin{equation*}
    T=uI_{n+1}+e(b,-2),\qquad
    T^\vee=uI_{n+1}+e^\vee(-b,-2u).
\end{equation*}

To check that these matrices describe the required local model, write $J_0=\operatorname{diag}(J,0)$ and let $Q=\operatorname{diag}(0_n,1)$. On the original crystal the action and the alternating pairing have matrices
\begin{equation*}
    \Pi=\begin{pmatrix}0&u^2I_{n+1}\\I_{n+1}&0\end{pmatrix},\qquad
    \begin{pmatrix}J_0&Q\\-Q&-u^2J_0\end{pmatrix}.
\end{equation*}
The first $n$ coordinates have the $\varpi_{E_v}$-modular pairing, and the last factor is self-dual. We choose the basis so that the filtration at the given CM point corresponds to $b=0$. In the dual basis, the lattice inclusion is
\begin{equation*}
    S=\begin{pmatrix}Q&u^2(I_{n+1}-Q)\\I_{n+1}-Q&Q\end{pmatrix}.
\end{equation*}
We define the Hodge filtrations on the two crystals to be the submodules generated by the columns of $\binom{-T}{I_{n+1}}$ and $\binom{-T^\vee}{I_{n+1}}$, respectively. Using $bJb^t=0$, direct multiplication gives
\begin{equation*}
    \begin{gathered}
    T^2=(T^\vee)^2=u^2I_{n+1},\qquad
    T^t(J_0+Q)=(J_0+Q)T^\vee,\\
    S\binom{-T}{I_{n+1}}
      =\binom{-T^\vee}{I_{n+1}}\bigl(Q-(I_{n+1}-Q)T\bigr).
    \end{gathered}
\end{equation*}
These identities show that the filtrations are stable under the $\mathcal O_{E_v}$-action and satisfy the conditions imposed by the polarization and the lattice inclusions. The remaining inclusion follows from periodicity. Moreover, the action of $\varpi_{E_v}$ on the Lie quotients has characteristic polynomial $(t-u)^n(t+u)$, as required by the signature condition.

Moreover, the filtrations above define a family in the naive local model over $R$. Since $R$ is flat over $\mathcal O_{\breve E_v}$ and the flat local model is the closure of the generic fiber, this family also lies in the flat local model. Modulo $u$, the matrix $-T^\vee$ becomes
\begin{equation*}
    -T^\vee=
    \begin{pmatrix}
        \frac12Jb^tb&0\\
        b&0
    \end{pmatrix}.
\end{equation*}
With alternating matrix $-J$, this is exactly the chart in \cite[(3.1.5)--(3.1.10)]{Smithling}, with parameters $b_1,\ldots,b_n$. The flat local model is smooth of relative dimension $n$ here by \cite[Theorem~7.5 and its proof]{RSZ2}. Thus its completed local ring and $R$ are power series rings over $\mathcal O_{\breve E_v}$ with $n$ parameters. The filtrations above define a morphism from $\operatorname{Spf}R$ to the completed local model. Since it induces an isomorphism on the relative cotangent spaces at $b=0$, the formal inverse function theorem shows that this morphism is an isomorphism. Consequently, by the local model diagram, $T$ and $T^\vee$ give the actions of $\varpi_{E_v}$ on the corresponding Lie algebras, and the CM point corresponds to $b=0$.

The modules can now be read from these matrices. The row $(b,-2)$ contains a unit, so $\operatorname{im}(T-u)=R\cdot e$ is free and equals its double dual. On the CM component its inclusion sends $e$ to $u$ times the negative Lie generator. For the dual action, the last coordinate of $e^\vee$ is $1$. If $(T^\vee+u)z=0$, its last coordinate and cancellation of $u$ give $z=e^\vee z_{n+1}$. Hence its kernel is $R\cdot e^\vee$, an invertible direct summand. This argument also holds after any base change on which $u$ is a nonzero divisor.

Finally, the lattice inclusion given by $S$ induces
\begin{equation*}
    d\lambda=
    \begin{pmatrix}uI_n-\frac12Jb^tb&0\\-b&1\end{pmatrix},\qquad
    d\lambda'=
    \begin{pmatrix}I_n&-Jb^t\\0&-u\end{pmatrix}.
\end{equation*}
Indeed, the first matrix equals $Q+T^\vee(I_{n+1}-Q)$, while the second matrix follows from $\lambda'\lambda=\iota_{\mathcal A}(u)$, since $d\lambda'\,d\lambda=T$. They satisfy
\begin{equation*}
    d\lambda(e)=u e^\vee,\qquad d\lambda'(e^\vee)=-e.
\end{equation*}
These identities prove \eqref{unit image modules}. In particular, the image sheaf is computed on the full chart and then restricted to the CM component.
\end{proof}

Lemma~\ref{unit image calculation} now allows us to compare the normal bundle of $\mathcal Y(x)$ with the Hodge bundle. For a nonzero $x$ whose divisor contains the generic CM point, the generic polarization gives the rational isomorphism
\begin{equation*}
    \kappa_x:\mathcal O(D_x)|_{\mathcal P}\dashrightarrow
       \mathcal L^{-1}|_{\mathcal P}.
\end{equation*}
The integral comparison below uses this fixed generic map.

\begin{proposition}\label{unit conormal comparison}
At $v\in S_0(W)$, the map $\varpi_{E_v}\kappa_x$ extends to an integral isomorphism. Thus the normalized finite degree of $\operatorname{div}(\kappa_x)$ above $v$ is $-\log N_v$, independently of $x$ and of the even dimension $n$.
\end{proposition}
\begin{proof}
The obstruction to lifting $\lambda x$ identifies the normal bundle of $D_x=\mathcal Y(x)$. We then compare this identification with $\kappa_x$ using $d\lambda'$.

Let $\mathcal J$ be the Cartier ideal of $D_x$. Since $D_x$ is relative Cartier by Lemma~\ref{unit horizontal divisor}, $\mathcal O/\mathcal J$ is flat over the base and $\mathcal J/\mathcal J^2$ is invertible over $\mathcal O/\mathcal J$. Hence the exact sequence
\begin{equation*}
    0\longrightarrow\mathcal J/\mathcal J^2
      \longrightarrow\mathcal O/\mathcal J^2
      \longrightarrow\mathcal O/\mathcal J\longrightarrow0
\end{equation*}
also gives flatness of $\mathcal O/\mathcal J^2$.

On this first infinitesimal neighborhood, the Grothendieck--Messing obstruction for $\lambda x$ is constructed as in the proof of \cite[Proposition~4.1]{Howard} and \cite[Lemma~2.39]{LL}, on the adic quotients followed by the inverse limit. The filtration $\Fil\mathcal A_0$ has $u$-eigenvalue $-u$. Since $u$ is a nonzero divisor on $\mathcal O/\mathcal J^2$, Lemma~\ref{unit image calculation} shows that the obstruction lies in
\begin{equation*}
    \operatorname{Hom}(\Fil\mathcal A_0,R\cdot e^\vee)\otimes\mathcal J/\mathcal J^2.
\end{equation*}
Moreover, $R\cdot e^\vee$ is a direct summand. Since both modules above have rank one, the lifting obstruction is locally given by a single equation. As its vanishing defines $D_x$, this equation generates $\mathcal J/\mathcal J^2$ on the first infinitesimal neighborhood. Dualizing the resulting isomorphism gives
\begin{equation}\label{unit divisor normal bundle}
    \mathcal O(D_x)|_{D_x}
      \simeq\operatorname{Hom}(\Fil\mathcal A_0,R\cdot e^\vee)|_{D_x}.
\end{equation}

Next, composition with $d\lambda'$ identifies $R\cdot e^\vee$ integrally with $R\cdot e$ by \eqref{unit image modules}. Its generic map is determined by
\begin{equation*}
    \lambda'\lambda x=x\iota_0(u).
\end{equation*}
Indeed, $\iota_0(u)$ acts as $-u$ on $\Fil\mathcal A_0$, so the induced map in \eqref{unit divisor normal bundle} is $-u\kappa_x$. As $R\cdot e=\mathcal I_{\mathcal A}^{\vee\vee}$, this proves that $u\kappa_x$ is an integral isomorphism.

Finally, since $u\kappa_x$ is an integral isomorphism, the order of $\kappa_x$ on each normalized CM component is $-\operatorname{ord}(u)$. As $u$ is a uniformizer of $E_v$, the normalized order is $-1$. Hence its arithmetic degree is $-\log N_v$ by \eqref{arithmetic normalization}.
\end{proof}

\subsubsection*{The restricted degree}

Now we return to the full contribution $i(x,\mathcal P)$ in \eqref{height decomposition}. Lemma~\ref{Hodge stabilization} and Proposition~\ref{unit conormal comparison} have already given its finite comparison with $\mathcal L^{-1}$. Combining the preceding comparisons with the metric in \eqref{central limit}, we obtain the following arithmetic adjunction formula, analogous to \cite[Theorem~9.3]{YZ} for CM points on Shimura curves. The additional terms at $S_0(W)$ arise from Proposition~\ref{unit conormal comparison}.

\begin{proposition}\label{restricted degree proposition}
For $0\ne x\in W^\perp(E)$ contributing to the CM term in \eqref{height decomposition},
\begin{equation}\label{adjunction integral equation}
    i(x,\mathcal P)+\log|q(x)|_f+dH_{n-1}
      =-\LL\cdot\mathcal P+\sum_{v\in S_0(W)}\log N_v.
\end{equation}
\end{proposition}
\begin{proof}
We apply the arithmetic degree formula to the rational comparison
\begin{equation*}
    \mathcal O(C_x)|_{\mathcal P}\longrightarrow
    \mathcal O(D_x)|_{\mathcal P}\dashrightarrow\mathcal L^{-1}|_{\mathcal P}.
\end{equation*}
By Lemma~\ref{Hodge stabilization} and Proposition~\ref{unit conormal comparison}, its normalized finite divisor degree is
\begin{equation*}
    (R_x,\mathcal P)_{\mathrm{fin}}-
       \sum_{v\in S_0(W)}\log N_v.
\end{equation*}
At each embedding of $E$, the calculation in \cite[Lemma~5.1]{Guo2} together with \eqref{central limit}, gives logarithmic norm $\frac12\log q(x)-\frac12H_{n-1}$. The product formula makes the sum of these norms $-\log|q(x)|_f-dH_{n-1}$. Combining the finite and archimedean degrees with the definition in \eqref{restricted divisor degree} proves \eqref{adjunction integral equation}.
\end{proof}

\subsubsection*{The correction at the two split places}

Recall that our Proposition~\ref{restricted degree proposition} computes the contribution of the full lifting divisor $D_x$. In \eqref{degenerate phi}, the summand $1_{\varpi_v\Lambda_v}$ imposes the stronger condition that $x/\varpi_v$ lift. We must therefore subtract the residual intersections removed by this condition. This is the last line of \eqref{height decomposition}, which we now compute.

Fix $v\in S_2$, and let $\mathfrak p_v$ be its prime ideal in $\mathcal O_F$. The subgroup $\mathcal A_0[\mathfrak p_v]$ means $\mathcal A_0[\mathfrak p_v\mathcal O_E]$ and includes the torsion at both places above $v$. The local function $1_{\varpi_v\Lambda_v}$ imposes
\begin{equation}\label{split kernel condition}
    x|_{\mathcal A_0[\mathfrak p_v]}=0.
\end{equation}
Indeed, this is the Tate module condition on the generic fiber. Over the integral base it is equivalent to the factorization
\begin{equation*}
    x:\mathcal A_0\longrightarrow\mathcal A_0/\mathcal A_0[\mathfrak p_v]\longrightarrow\mathcal A.
\end{equation*}
The factorization is unique as a morphism of fppf sheaves. By construction, it is compatible with the finite forgetful map. Indeed, multiplication by $\varpi_v$ identifies $\mathcal A_0/\mathcal A_0[\mathfrak p_v]$ with $\mathcal A_0$ on the relative factor at $v$. Thus the factorization above is equivalent to requiring $x/\varpi_v$ to extend as a homomorphism.

\begin{proposition}\label{split arithmetic correction}
For the Schwartz function \eqref{degenerate phi}, the arithmetic contribution along the CM cycle differs from the sum formed with the full lifting divisors by
\begin{equation}\label{split arithmetic correction formula}
    -2\log N_v\sum_{x\in W^\perp(E)\setminus\{0\}}1_{\varpi_v\Lambda_v^\perp}(x)1_{\varpi_v^2\mathcal O_{F_v}^\times}(q(x))\,r(g^v)\Phi^v(x,0),
\end{equation}
for $g_f\in K_f$.
\end{proposition}

\begin{proof}
This correction is actually the intersection with the part of the full lifting divisor removed by \eqref{split kernel condition}. At one place of $E$ above $v$, the two equations are given by \cite[Proposition~5.9]{Guo2}. Indeed, on the support of the first summand of \eqref{degenerate phi}, the homomorphism $x/\varpi_v$ has unit norm, so both its split components are primitive. By the moduli interpretation of the RSZ integral model and the ordinary deformation calculation in \cite[Proposition~5.9]{Guo2}, we may choose a Serre--Tate coordinate $z$ near a geometric point of the CM cycle such that
\begin{equation*}
    x/\varpi_v\text{ lifts}\ \Longleftrightarrow\ z=0, \qquad x\text{ lifts}\ \Longleftrightarrow\ [\varpi_v]_G(z)=0,
\end{equation*}
where $G$ is a formal $\mathcal O_{F_v}$-module of height one.
Both divisors are flat. Consequently they equal their horizontal closures and give the finite admissible extensions by \cite[Lemma~A.2.4]{Qiu}.

Since both terms use the Green function associated with $x$, their regularized metrics agree along $z=0$. For the finite intersection, write
\begin{equation*}
    [\varpi_v]_G(z)=zG_v(z),\qquad G_v(0)=\varpi_v.
\end{equation*}
Thus the complete lifting divisor has the additional factor $G_v(z)$, whose restriction to the CM cycle equals $\varpi_v$. Thus, on each normalized CM component with valuation ring $R$, the intersection with the divisor $G_v(z)=0$ has length
\begin{equation*}
    \operatorname{length}_R(R/(\varpi_v))
    =\operatorname{ord}_R(\varpi_v).
\end{equation*} 
After normalization, the finite intersection decreases by $\log N_v$ at each of the two places of $E$ above $v$, and hence by $2\log N_v$ in total.

Finally, the factorization through $\mathcal A_0/\mathcal A_0[\mathfrak p_v]$ preserves the finite maps and stabilizers. The projection formula therefore gives the multiplicities in \eqref{arithmetic normalization}, as in \cite[Lemma~8.4 and Proposition~8.5]{YZ}. Note that only the first summand of \eqref{degenerate phi} is affected. Thus, multiplying the correction $-2\log N_v$ by its characteristic functions and the factors away from $v$ and summing over $x$ gives \eqref{split arithmetic correction formula}. This local correction is Schwartz and is fixed by $K_f$, so the formula holds for $g_f\in K_f$.
\end{proof}

The local function in \eqref{split arithmetic correction formula} is $c_{\Phi_v}(1,x)$ in \eqref{split zero derivative}, which occurs with a minus sign in the derivative series \eqref{analytic formula}. Hence we conclude that the additional arithmetic and analytic terms at each modified split place cancel for every $x$.

\section{Comparison of the two series}\label{comparison section}

We now add the derivative formula in Theorem~\ref{derivative decomposition theorem} and the height formula in Theorem~\ref{height decomposition theorem}. The local results of Section~\ref{local intersection subsection} combine the proper terms into theta series. Moreover, Proposition~\ref{restricted degree proposition} gives the remaining restricted degree, and Proposition~\ref{split arithmetic correction} cancels the two split corrections. We will combine all these computations into Proposition~\ref{direct comparison proposition}. Then we apply the Key lemma to determine the height and deduce both main theorems.

\subsection{The difference series}

Following the strategy of \cite{YZZ,YZ,Yuan,Guo3}, we consider the following difference series
\begin{equation}\label{difference series definition}
    \begin{aligned}
    \mathcal D(g,\Phi)&=\Pr I'(0,g,\Phi)+\mathcal H(g,\Phi)\\
    &=\Pr I'(0,g,\Phi)+\left(\widehat{\mathcal Z}_*(g,\Phi)-\frac{\deg_L Z_*(g,\Phi)}{\deg_L X_U}\LL\right)\cdot(\mathcal P-\widehat\xi).
    \end{aligned}
\end{equation}
Proposition~\ref{cuspidal height pairing} makes \eqref{difference series definition} a holomorphic cusp form and removes its geometric degree and Hodge pairing. Thus
\begin{equation*}
    \mathcal D(g,\Phi)=\Pr I'(0,g,\Phi)+\widehat{\mathcal Z}_*(g,\Phi)\cdot\mathcal P,
    \qquad g_f\in K_f.
\end{equation*}

Consider the full theta series defined in Section~\ref{ASW introduction}. Since $[E^1]$ is compact, the archimedean Gaussians are invariant under the torus $E^1$, and the finite functions have open stabilizers, the normalized integral over $[E^1]$ is a finite weighted sum of theta series. We choose $K_f$ to fix all finite functions in this sum.

At finite component $1$, the sum \eqref{Omega definition} satisfies
\begin{equation}\label{theta factorization}
    \Omega_\Phi(g_\infty,1)=r_{\BW_\infty}(g_\infty)\Phi_{\BW,\infty}(0)\,
    \theta(g_\infty,1,\Phi^\perp).
\end{equation}
Here $\Phi_{\BW,\infty}$ is the standard Gaussian. The zero vector contributes nothing because $\Phi_f(0)=0$. This additional Gaussian factor will separate $\Omega_\Phi$ from automorphic forms.

\begin{proposition}\label{direct comparison proposition}
For every $W$ in Theorem~\ref{CM height theorem} and $g_f\in K_f$,
\begin{equation}\label{direct comparison}
    \begin{aligned}
    \mathcal D(g,\Phi)={}&-\sum_{\substack{v\nmid\infty\\v\text{ nonsplit}}}
    \avint_{[E^1]}\theta\bigl(g,\tau,\bar k_{\Phi_v}\otimes\Phi^v\bigr)\,\dd\tau\\
    &-\left(\LL\cdot\mathcal P+2\frac{L_f'(0,\eta)}{L_f(0,\eta)}
       -\sum_{v\nmid\infty}a_v(c_v)\log N_v\right)\Omega_\Phi(g).
    \end{aligned}
\end{equation}
The local functions in the theta integrals are Schwartz and zero at almost every place.
\end{proposition}
\begin{proof}
The proof is straight-forward. We add \eqref{analytic formula} and \eqref{height decomposition}. The archimedean terms cancel by \eqref{archimedean proper comparison}. At a finite nonsplit place, we complete the sum with coefficient $-\bar k_{\Phi_v}$ to a full theta series. This adds $\bar k_{\Phi_v}(1,x)$ on $W^\perp$, which is zero at $S_0(W)$ by Proposition~\ref{unit local matching}.

By \eqref{adjunction integral equation}, the logarithmic norm and harmonic terms cancel. Moreover, $i(x,\mathcal P)$ already includes the finite intersections with the residual branches. Since $\delta_f(g_f)=1$, these terms contribute the following to the coefficient of $\Omega_\Phi$:
\begin{equation*}
    -\LL\cdot\mathcal P-c_1+\sum_{v\in S_0(W)}\log N_v.
\end{equation*}
At the two modified split places, Proposition~\ref{split arithmetic correction} cancels the zero derivative in \eqref{split zero derivative}. At the other places, the proper comparisons give
\begin{equation}\label{all finite boundaries}
    \begin{aligned}
    &\bar k_{\Phi_v}(1,x)+m^0_{\Phi_v}(x)\log N_v-c_{\Phi_v}(1,x)\\
    &\quad=\begin{cases}
        \bigl(v(d_v)+a_v(c_v)\bigr)\Phi_v(x,0)\log N_v,
             &v\text{ unramified in }E/F,\\
        \bigl(1+a_v(c_v)\bigr)\Phi_v(x,0)\log N_v,
             &v\text{ ramified},\ r_v\ge1,\\
        0,&v\in S_0(W).
    \end{cases}
    \end{aligned}
\end{equation}
Here $\bar k=0$ at a split place. The first two cases are Propositions~\ref{inert general matching}, \ref{split general matching}, \ref{ramified general matching} and Corollary~\ref{inert different boundary}, while the last is Proposition~\ref{unit local matching}.

The sum of \eqref{all finite boundaries} together with the conormal contribution from $S_0(W)$ is
\begin{equation*}
    \log d_F+\log d_{E/F}+\sum_{v\nmid\infty}a_v(c_v)\log N_v.
\end{equation*}
The first two terms cancel the discriminant part of $c_1$. These local identities are compatible with the action of $E^1(\mathbb A)$ and hold for $g_f\in K_f$. Thus, taking the normalized integral over $[E^1]$ with the multiplicities defining $\mathcal P$, we finish the proof.
\end{proof}

\subsection{Proofs of the main theorems}

The comparison of generating series follows the method of \cite{YZZ}. The pseudo-theta series and the Key lemma used to separate their contributions were introduced in \cite[Section~6, Lemma~6.1]{YZ}; \cite[Lemma~2.3]{Yuan} extends the argument to include pseudo-Eisenstein series. We apply the unitary form in \cite[Lemma~2.4]{Guo1}. Here the local differences have already been completed to theta series, so the lemma determines the coefficient of $\Omega_\Phi$ in \eqref{direct comparison}.

\begin{proof}[Proof of Theorem~\ref{CM height theorem}]
We choose a sufficiently divisible positive integer $b$ such that
\begin{equation*}
    \gamma=\begin{pmatrix}1&0\\b&1\end{pmatrix}\in\mathrm{SU}(1,1)(F),
    \qquad\gamma_f\in K_f.
\end{equation*}
The automorphic forms in \eqref{direct comparison}, and the theta series on $W^\perp$, have the same values at $(g_\infty,1)$ and $(\gamma_\infty g_\infty,1)$. Subtracting and using \eqref{theta factorization} gives
\begin{equation*}
    \begin{aligned}
    0={}&\left(\LL\cdot\mathcal P+2\frac{L_f'(0,\eta)}{L_f(0,\eta)}
       -\sum_{v\nmid\infty}a_v(c_v)\log N_v\right)
       \theta(g_\infty,1,\Phi^\perp)\\
    &\cdot\bigl(r_{\BW_\infty}(\gamma_\infty g_\infty)\Phi_{\BW,\infty}(0)
       -r_{\BW_\infty}(g_\infty)\Phi_{\BW,\infty}(0)\bigr).
    \end{aligned}
\end{equation*}
For $g_\infty=\prod_{v\mid\infty}m(a)$, $a>0$, the ratio of the absolute values of the two Gaussian factors is $(1+b^2a^4)^{-d/2}<1$. Their difference is nonzero. Lemma~\ref{nonzero rational theta coefficient} makes the theta factor nonzero for sufficiently small $a$. Then the coefficient in parentheses is zero.
\end{proof}

\begin{proof}[Proof of Theorem~\ref{ASW main theorem}]
We substitute Theorem~\ref{CM height theorem} into \eqref{direct comparison}. This makes the coefficient of $\Omega_\Phi$ zero and gives \eqref{ASW main identity}.
\end{proof}

At the end of this paper, we give a remark to explain why the main Theorem~\ref{CM height theorem} remains valid when $F=\mathbb{Q}$.

\begin{remark}\label{F=Q remark}
    Theorem~\ref{CM height theorem} also holds for $F=\Q$, with the other assumptions unchanged and $S_0(W)=\varnothing$. For $n=1$, the explicit comparison in \cite[Section~3]{Guo2} identifies differences of normalized CM heights with twice the corresponding differences of stable Faltings heights of CM elliptic curves. Indeed, the line bundles from the fixed auxiliary CM datum cancel, and the factor two comes from the degree normalization over $E$. By Proposition~\ref{local CM order}, the quadratic order is $\mathbb Z+f\mathcal O_E$ where $f=\prod_p p^{c_p}$. Thus \cite{NT}, together with the maximal-order formula in \cite[Theorem~1.2 and Section~1.4]{Yuan}, proves the assertion for curves.

    For $n\ge2$, let $T$ be any finite set of primes containing those ramified in $E/\Q$ and those with $c_p>0$, and fix $0\ne w\in W$. The rank-two summands in Lemma~\ref{all place line types} are hyperbolic, since $S_0(W)=\varnothing$. By the Hasse principle and weak approximation, we choose $x\in W^\perp$ with $q(x)=-q(w)$ that satisfies the lattice conditions in Section \ref{notation subsection} at every $p\in T$. Then $W\oplus Ex$ is a hyperbolic plane. We use the auxiliary curve construction and the compatibility of normalized CM cycles and Hodge line bundles in \cite[Section~3.1]{Guo2}, together with Lemma~\ref{Hodge stabilization}. The comparison is an isometry at infinity and an integral isomorphism outside a finite set $S$ disjoint from $T$. The curve formula just proved, with the same integers $c_p$ outside $S$, consequently gives
    \begin{equation*}
       \LL\cdot\mathcal P+2\frac{L_f'(0,\eta)}{L_f(0,\eta)}-\sum_p a_p(c_p)\log p\in\sum_{p\in S}\Q\log p.
    \end{equation*}
    We now add the primes in $S$ to $T$ and choose another auxiliary curve. This gives the same inclusion for a new finite set $S'$ disjoint from $S$. The linear independence over $\Q$ of the logarithms of distinct primes then proves the height formula.
\end{remark}


\begin{thebibliography}{AGHMP}

\bibitem[AGHMP]{AGHMP}
F. Andreatta, E.~Z. Goren, B. Howard, and K. Madapusi Pera. Faltings heights of abelian varieties with complex multiplication. Ann. of Math. (2) \textbf{187} (2018), 391--531.

\bibitem[BHKRY]{BHKRY}
J.~H. Bruinier, B. Howard, S.~S. Kudla, M. Rapoport, and T. Yang. Modularity of generating series of divisors on unitary Shimura varieties II: arithmetic applications. Ast\'erisque \textbf{421} (2020), 127--186.

\bibitem[BHY]{BHY}
J.~H. Bruinier, B. Howard, and T. Yang. Heights of Kudla--Rapoport divisors and derivatives of $L$-functions. Invent. Math. \textbf{201} (2015), 1--95.

\bibitem[BY]{BY}
J.~H. Bruinier and T. Yang. Faltings heights of CM cycles and derivatives of $L$-functions. Invent. Math. \textbf{177} (2009), 631--681.

\bibitem[Chen]{Chen}
R.~C. Chen. Co-rank 1 arithmetic Siegel--Weil I: Local non-Archimedean. \href{https://arxiv.org/abs/2405.01426}{arXiv:2405.01426}, 2024.

\bibitem[Col]{Colmez}
P. Colmez. P\'eriodes des vari\'et\'es ab\'eliennes \`a multiplication complexe. Ann. of Math. (2) \textbf{138} (1993), 625--683.

\bibitem[GS]{GS}
H. Gillet and C. Soul\'e. Arithmetic intersection theory. Publ. Math. Inst. Hautes \'Etudes Sci. \textbf{72} (1990), 93--174.

\bibitem[Guo1]{Guo1}
Z. Guo. Modular Heights of Unitary Shimura Varieties I: Derivatives of Eisenstein Series. Revised manuscript, 2026. \href{https://arxiv.org/abs/2509.24363}{arXiv:2509.24363}.

\bibitem[Guo2]{Guo2}
Z. Guo. Modular Heights of Unitary Shimura Varieties II: Arithmetic Generating Series of Divisors. Revised manuscript, 2026. \href{https://arxiv.org/abs/2606.18579}{arXiv:2606.18579}.

\bibitem[Guo3]{Guo3}
Z. Guo. Modular Heights of Unitary Shimura Varieties III: Proof of the Main Theorem. Revised manuscript, 2026. \href{https://arxiv.org/abs/2606.19158}{arXiv:2606.19158}.

\bibitem[How]{Howard}
B. Howard. Linear invariance of intersections on unitary Rapoport--Zink spaces. Forum Math. \textbf{31} (2019), 1265--1281.

\bibitem[KR]{KR}
S. Kudla and M. Rapoport. Special cycles on unitary Shimura varieties I. Unramified local theory. Invent. Math. \textbf{184} (2011), 629--682.

\bibitem[KRY1]{KRY1}
S.~S. Kudla, M. Rapoport, and T. Yang. Derivatives of Eisenstein series and Faltings heights. Compos. Math. \textbf{140} (2004), no.~4, 887--951.

\bibitem[KRY2]{KRY}
S. Kudla, M. Rapoport, and T. Yang. \emph{Modular Forms and Special Cycles on Shimura Curves}. Annals of Mathematics Studies \textbf{161}, Princeton University Press, 2006.

\bibitem[Kud]{Kudla}
S.~S. Kudla. Special cycles and derivatives of Eisenstein series. In \emph{Heegner Points and Rankin $L$-Series}, Math. Sci. Res. Inst. Publ. \textbf{49}, Cambridge University Press, 2004, 243--270.

\bibitem[Liu1]{Liu1}
Y. Liu. Arithmetic theta lifting and $L$-derivatives for unitary groups, I. Algebra Number Theory \textbf{5} (2011), 849--921.

\bibitem[Liu2]{Liu2}
Y. Liu. Arithmetic theta lifting and $L$-derivatives for unitary groups, II. Algebra Number Theory \textbf{5} (2011), 923--1000.

\bibitem[LL]{LL}
C. Li and Y. Liu. Chow groups and $L$-derivatives of automorphic motives for unitary groups, II. Forum Math. Pi \textbf{10} (2022), e5. We use the numbering of the published article.

\bibitem[LMZ]{LMZ}
Y. Luo, A. Mihatsch, and Z. Zhang. Unitary Shimura varieties at ramified primes and arithmetic transfer. \href{https://arxiv.org/abs/2504.17484v2}{arXiv:2504.17484v2}, 2025.

\bibitem[Mih]{Mih}
A. Mihatsch. An arithmetic transfer identity. Manuscripta Math. \textbf{150} (2016), 1--19. The local numbering used here is that of \href{https://arxiv.org/abs/1406.7764v3}{arXiv:1406.7764v3}.

\bibitem[NT]{NT}
Y. Nakkajima and Y. Taguchi. A generalization of the Chowla--Selberg formula. J. Reine Angew. Math. \textbf{419} (1991), 119--124.

\bibitem[OT]{OT}
T. Oda and M. Tsuzuki. Automorphic Green functions associated with the secondary spherical functions. Publ. Res. Inst. Math. Sci. \textbf{39} (2003), 451--533.

\bibitem[Qiu]{Qiu}
C. Qiu, with an appendix by Y. Xu. Modularity of arithmetic special divisors for unitary Shimura varieties. Forum Math. Sigma \textbf{13} (2025), e111.

\bibitem[RSZ1]{RSZ1}
M. Rapoport, B. Smithling, and W. Zhang. On the arithmetic transfer conjecture for exotic smooth formal moduli spaces. Duke Math. J. \textbf{166} (2017), 2183--2336.

\bibitem[RSZ2]{RSZ2}
M. Rapoport, B. Smithling, and W. Zhang. Regular formal moduli spaces and arithmetic transfer conjectures. Math. Ann. \textbf{370} (2018), 1079--1175.

\bibitem[RSZ3]{RSZ3}
M. Rapoport, B. Smithling, and W. Zhang. On Shimura varieties for unitary groups. Pure Appl. Math. Q. \textbf{17} (2021), 773--837.

\bibitem[Smi]{Smithling}
B. Smithling. On the moduli description of local models for ramified unitary groups. Int. Math. Res. Not. IMRN \textbf{2015} (2015), 13493--13532.

\bibitem[Tate]{Tate}
J.~T. Tate. $p$-divisible groups. In \emph{Proceedings of a Conference on Local Fields, Driebergen, 1966}, Springer, 1967, 158--183.

\bibitem[Wew1]{Wew1}
S. Wewers. Canonical and quasi-canonical liftings. In \emph{ARGOS Seminar on Intersections of Modular Correspondences}, Ast\'erisque \textbf{312} (2007), 67--86.

\bibitem[Wew2]{Wew2}
S. Wewers. An alternative approach using ideal bases. In \emph{ARGOS Seminar on Intersections of Modular Correspondences}, Ast\'erisque \textbf{312} (2007), 171--177.

\bibitem[Yao]{Yao}
H. Yao. A Kudla--Rapoport formula for exotic smooth models of odd dimension. \href{https://arxiv.org/abs/2404.14431v1}{arXiv:2404.14431v1}, 2024.

\bibitem[Yuan]{Yuan}
X. Yuan. Modular Heights of Quaternionic Shimura Curves. \href{https://arxiv.org/abs/2205.13995v3}{arXiv:2205.13995v3}, 2024.

\bibitem[YZ]{YZ}
X. Yuan and S.-W. Zhang. On the averaged Colmez conjecture. Ann. of Math. (2) \textbf{187} (2018), 533--638; erratum, Ann. of Math. (2) \textbf{198} (2023), 867--878.

\bibitem[YZZ]{YZZ}
X. Yuan, S.-W. Zhang, and W. Zhang. The Gross--Zagier Formula on Shimura Curves. Annals of Mathematics Studies \textbf{184}, Princeton University Press, 2013.

\bibitem[Zhao1]{Zhao1}
R. Zhao. Heights of special points on quaternionic Shimura varieties. Preprint, 2023. \href{https://arxiv.org/abs/2309.08886}{arXiv:2309.08886}.

\bibitem[Zhao2]{Zhao2}
R. Zhao. Towards the Colmez conjecture. Acta Arith. \textbf{223} (2026), 207--226.

\end{thebibliography}
\end{document}